\documentclass[reqno,11pt]{amsart}
\usepackage{amsmath,amsrefs,color}
\usepackage{geometry} 
\usepackage{graphicx,picture}
\usepackage[colorinlistoftodos]{todonotes}
\numberwithin{equation}{section}
\usepackage[normalem]{ulem}
\newtheorem{theorem}{Theorem}[section]
\newtheorem{defi}[theorem]{Definition}
\newtheorem{proposition}[theorem]{Proposition}
\newtheorem{lemma}[theorem]{Lemma}
\newtheorem{cor}[theorem]{Corollary}
\newtheorem{rem}[theorem]{Remark}
\newtheorem{ex}[theorem]{Example\/}

\newcommand{\R}{\mathbb{R}}

\numberwithin{equation}{section}

\begin{document}

\title[Anisotropic elliptic equations with $L^1$ data]{Anisotropic elliptic equations with gradient-dependent lower order terms and $L^1$ data}

\author{Barbara Brandolini}
\address[Barbara Brandolini]{Dipartimento di Matematica e Informatica, Universit\`a degli Studi di Palermo, via Archirafi 34, 90123 Palermo, Italy {\em (current address)}; 
Dipartimento di Matematica e Applicazioni ``R. Caccioppoli'', Universit\`a degli Studi di Napoli Federico II, Complesso Monte S. Angelo - via Cintia, 80126 Napoli, Italy}
\email{barbara.brandolini@unipa.it}

\author{Florica C. C\^{\i}rstea}
\address[Florica C. C\^irstea]{School of Mathematics and Statistics, The University of Sydney,
	NSW 2006, Australia}
\email{florica.cirstea@sydney.edu.au}

\date{}

%\dedicatory{Dedicated to Professor Neil S. Trudinger on the occasion of his 80th birthday}

\begin{abstract}
We prove the existence of a  weak solution for a general class of Dirichlet anisotropic elliptic problems such as $\mathcal Au+\Phi(x,u,\nabla u)=\mathfrak{B}u+f$ in $\Omega$, where $\Omega$ is a bounded open subset of $\mathbb R^N$ and $f\in L^1(\Omega)$ is arbitrary. The principal part is a divergence-form nonlinear anisotropic operator $\mathcal A$, the prototype of which is $\mathcal A u=-\sum_{j=1}^N \partial_j(|\partial_j u|^{p_j-2}\partial_j u)$ with $p_j>1$ for all $1\leq j\leq N$ and $\sum_{j=1}^N (1/p_j)>1$. As a novelty in this paper, 
	our lower order terms involve a new class of operators $\mathfrak B$ such that $\mathcal{A}-\mathfrak{B}$ is bounded, coercive and pseudo-monotone from $W_0^{1,\overrightarrow{p}}(\Omega)$ into its dual, as well as a gradient-dependent nonlinearity  $\Phi$ with an  ``anisotropic natural growth" in the gradient and a good sign condition.
\end{abstract}

\maketitle

\section{INTRODUCTION AND MAIN RESULT}

\subsection{Setting of the problem}
In their famous book \cite{GT}, 
Gilbarg and Trudinger captured  
the astonishing achievements in the theory of nonlinear elliptic second order partial differential equations. For recent developments of fully nonlinear equations and their applications to 
optimal transportation and conformal 
geometry, see e.g., Trudinger \cite{Tr1,Tr2}. 

A quasilinear operator is not always the differential of a functional of the calculus of
variations. What makes it possible to go further than the calculus of variations in
the convex case is the abstract concept of monotone operator and, more generally, of pseudo-monotone
operator. Several papers \cite{BBoc,BBM,BOG,BMP,BGal} deal with nonlinear elliptic problems in a bounded open subset $\Omega$ of $\mathbb R^N$ involving coercive, bounded, continuous and 
pseudo-monotone 
Leray--Lions type operators from $W_0^{1,p}(\Omega)$ into its dual $W^{-1,p'}(\Omega)$, where $1<p<\infty$ and $p'=p/(p-1)$ is the conjugate exponent of $p$. 
The prototype model of such an operator is the $p$-Laplacian $\Delta_p u={\rm div}\, (|\nabla u|^{p-2}\nabla u)$.  
The techniques developed in these papers accommodate for a lower-order term
$g(x,u,\nabla u)$ with a ``natural growth" in the gradient $|\nabla u|$ and without any restriction of its growth in $|u|$. Either $f\in L^1(\Omega)$ or $h\in W^{-1,p'}(\Omega)$ could be included because of the ``sign-condition" on $g$ (that is, $g(x,t,\xi)\,t\geq 0$ for a.e. $x\in \Omega$ and all $(t,\xi)\in \mathbb R\times \mathbb R^N$). For related works, we refer to \cite{AFM,AM,BMMP1,FMe,FM,GMP}.

In this paper, we expand the above research program into the anisotropic arena by providing a suitable general framework under which for {\em every} $f\in L^1(\Omega)$, we prove the existence of a weak solution to Dirichlet anisotropic elliptic problems such as 
\begin{equation} \label{eq101}
\left\{  \begin{aligned} 
&\mathcal A u + \Phi(x,u,\nabla u)+\Theta(x,u,\nabla u)=\mathfrak{B}u+
f\quad \mbox{in } \Omega,\\ 
& u\in W_0^{1,\overrightarrow{p}}(\Omega), \quad \Phi(x,u,\nabla u)\in L^1(\Omega). 
\end{aligned} \right.
\end{equation}   
Here, and henceforth, $\Omega$ is a bounded, open subset of $\mathbb R^N$ ($N\ge 2$). We impose no smoothness condition on $\partial\Omega$. 
We assume throughout that
\begin{equation} \label{IntroEq0}
1<p_j\leq p_{j+1}<\infty\ \mbox{for every }1\leq j\leq N-1 \quad  \text{and}\quad 
p<N,
\end{equation}
where $p=N/\sum_{j=1}^N (1/p_j)$ is the harmonic mean of $p_1,\ldots, p_N$.
Let $\nabla u=(\partial_1 u,\ldots,\partial_N u) $ be the gradient of $u$.  
Let $W_0^{1,\overrightarrow{p}}(\Omega)$ be the closure of   
$C_c^\infty (\Omega)$ (the set of smooth functions with compact support in $\Omega$)   
with respect to the norm 
$ \|u\|_{W_0^{1,\overrightarrow{p}}(\Omega)}=\sum_{j=1}^N \|\partial_j u\|_{L^{p_j}(\Omega)}$.    
We use $W^{-1,\overrightarrow{p}'}(\Omega)$ to denote the dual of $W_0^{1,\overrightarrow{p}}(\Omega)$ and $\langle \cdot, \cdot \rangle$ for 
the duality between $W^{-1,\overrightarrow{p}'}(\Omega)$ and 
$W_0^{1,\overrightarrow{p}}(\Omega)$.
The prototype for $\mathcal A$ is the anisotropic $\overrightarrow{p}$-Laplacian, namely, 
\begin{equation}  \label{lapal}
\mathcal A u=-\sum_{j=1}^N \partial_j(|\partial_j u|^{p_j-2}\partial_j u),
\end{equation}
%where we assume that
 (see \eqref{form} and \eqref{ellip}). 
The model for $\Phi$ in \eqref{eq101} is as follows
\begin{equation}  
\label{exa} 
\Phi(u,\nabla u)= \left(\sum_{j=1}^N |\partial_j u|^{p_j}+1\right)|u|^{m-2}u+\sum_{j=1}^N \mathfrak{b}_j |\partial_j u|^{q_j}\,|u|^{\theta_j-2}u , 
\end{equation}
where $\mathfrak{b}_j\geq 0$ and $0\leq q_j<p_j$, while $m,\theta_j>1$ for all $1\leq j\leq N$ are arbitrary (see \eqref{cond1} and \eqref{info}).  
We assume throughout that 
$\Theta(x,t,\xi): \Omega\times \mathbb R\times \mathbb R^N\to  \mathbb R$ is a Carath\' eodory function 
(that is, measurable on $\Omega$ for every $(t,\xi)\in  \mathbb R\times \mathbb R^N$ and continuous in 
$t,\xi$ for a.e. $x\in \Omega$) and there exists a constant $C_\Theta>0$ such that
\begin{equation} \label{newlab0}
|\Theta(x,t,\xi)|\leq C_\Theta \quad \mbox{
	for a.e. } x\in \Omega\ \mbox{and for all } (t,\xi)\in  \mathbb R\times \mathbb R^N.
\end{equation} 
Furthermore, our problem \eqref{eq101} features a new class of operators $\mathfrak{B}$ as follows.

\begin{defi} \label{doi}
	{\rm Let \eqref{form} and \eqref{ellip} hold. By $\mathfrak{BC}$ we denote the class of all {\em bounded} operators
		$\mathfrak B$ from $W_0^{1,\overrightarrow{p}}(\Omega)$ into  $W^{-1,\overrightarrow{p}'}(\Omega)$ satisfying the following two properties: 
		
		$(P_1)$ The operator $\mathcal A-\mathfrak{B} $ from $W_0^{1,\overrightarrow{p}}(\Omega)$ into  $W^{-1,\overrightarrow{p}'}(\Omega)$ is {\em coercive} (see Definition~\ref{nsr}). 
		 		
		$(P_2)$ If $u_\ell  \rightharpoonup u$ and $v_\ell  \rightharpoonup v$ (weakly) in $W_0^{1,\overrightarrow{p}}(\Omega)$ as $\ell\to \infty$, then
		\begin{equation*} \label{pp2} 
		\lim_{\ell\to \infty} 	\langle \mathfrak{B} u_\ell,v_\ell\rangle= \langle \mathfrak{B} u, v \rangle.
		\end{equation*}}
	Let $\mathfrak{BC}_+$ be the class of operators in $\mathfrak{BC}$ satisfying the extra condition

	$(P_3)$ For $\nu_0>0$ in the coercivity condition of \eqref{ellip} and each $k>0$, it holds 
	\begin{equation} \nu_0 \sum_{j=1}^N \|\partial_j u\|_{L^{p_j}(\Omega)}^{p_j}-\langle \mathfrak{B}u,T_k u \rangle
	\to \infty \ \mbox{as }	 \|u\|_{W_0^{1,\overrightarrow{p}}(\Omega)}\to \infty.	
	\end{equation}
\end{defi}
\noindent We use $T_k$ for the truncation at height $k$, see \eqref{trunc}. 
Unlike $\mathcal A$, the operator $-\mathfrak{B}$ is not coercive in general. 
Our assumption $(P_2)$ is reminiscent of $(iii)$ in the Hypothesis $(II)$ of Theorem~1 in the celebrated paper \cite{LL} by Leray and Lions. 
Every operator satisfying $(P_2)$ is strongly continuous (see Lemma~\ref{strong0}) and pseudo-monotone 
(cf. \cite[p. 586]{ze}).

In Example~\ref{exg} we use that $p^\ast=Np/(N-p)$ is the critical exponent for the embedding $W_0^{1,\overrightarrow{p}}(\Omega) \hookrightarrow L^r(\Omega)$ (see Remark~\ref{an-sob} in the Appendix). For any $r>1$, let $r'=r/(r-1)$. %We give examples of operators $\mathfrak{B}$ in the class $\mathfrak{BC}_+$.} 

%We conclude this subsection by {\color{blue} giving examples} of operators $\mathfrak{B}$ {\color{blue} in the class $\mathfrak{BC}_+$ introduced in} Definition \ref{doi}.
\begin{ex} \label{exg}
	{\rm Let $F\in L^{(p^\ast)'}(\Omega)$ and $h,\widetilde h\in W^{-1,\overrightarrow{p}'}(\Omega)$ be arbitrary. 
		Let $\rho,\alpha_k\in \mathbb R$ for $0\leq k\leq 4$. 
		For every $u\in W_0^{1,\overrightarrow{p}}(\Omega)$,  we define 
		
		(1) $\mathfrak{B} u =h$;
		
		(2) $\mathfrak{B} u = F+\rho\,|u|^{\vartheta-2}u$ with $1 < \vartheta <p$ if $\rho>0$ and $1<\vartheta<p^\ast$ if $\rho<0$;
		
		(3) $\mathfrak{B} u =\left(\alpha_0+\alpha_1 \|u\|^{\mathfrak{b}_1}_{L^r(\Omega)} +\alpha_2
		|\langle \widetilde h,u\rangle|^{\mathfrak{b_2}} \right) \left(
		\alpha_3 h+ \alpha_4  F  \right)$,
		where $ r\in [1,p^\ast)$; we take $ \mathfrak{b}_1\in (0,p/p_1')$ and $ \mathfrak{b}_2\in (0,p_1-1)$  if  $\alpha_3\not=0$; $ \mathfrak{b}_1\in (0,p-1)$ and $\mathfrak{b}_2\in (0,p_1/p')$ if $\alpha_3=0$;
		
		(4) $\mathfrak{B} u = -\sum_{j=1}^N \partial_j\left(\beta_j(x)+|u|^{\sigma_j-1}u\right)$, 
		where $ \beta_j\in L^{p_j'}(\Omega)$ and $0<\sigma_j<p/p_j'$ 
		for every $1 \le j \le N$.	
				
				In each of these situations, $\mathfrak{B}$ belongs to the class $\mathfrak{BC}_+$.}
\end{ex}

%The brackets $\langle \cdot,\cdot \rangle$ indicate the duality between $W^{-1,\overrightarrow{p}'}(\Omega)$ and $W_0^{1,\overrightarrow{p}}(\Omega)$. 

%For any $\mathfrak B$ in the class $\mathfrak{BC}$, the operator $\mathcal A-\mathfrak{B}:W_0^{1,\overrightarrow{p}}(\Omega)\to W^{-1,\overrightarrow{p}'}(\Omega)$ is coercive, bounded and pseudo-monotone.  

\subsection{Main result}
%Under the hypotheses \eqref{ellip}--\eqref{info}, we prove in Theorem~\ref{nth2} that \eqref{eq101} has a solution for every $f\in L^1(\Omega)$ and $\mathfrak B$ in the class $\mathfrak{BC}_+$. 
By a {\em solution} of \eqref{eq101} we mean any function $u\in  W_0^{1,\overrightarrow{p}}(\Omega)$ such that $\Phi(x,u,\nabla u)\in L^1(\Omega)$ and, for every
$v\in W_0^{1,\overrightarrow{p}}(\Omega)\cap L^\infty(\Omega)$,  
\begin{equation} \label{sesim}
\langle \mathcal Au,v\rangle +
\int_\Omega \Phi(x,u,\nabla u)\,v\,dx +
\int_\Omega \Theta(x,u,\nabla u)\,v\,dx=\langle \mathfrak{B} u,v\rangle +\int_\Omega f\,v\,dx.
\end{equation}

Under the assumptions in Section~\ref{asb}, the main advance in this paper is the following.

\begin{theorem} \label{nth2} 
Let \eqref{IntroEq0}, \eqref{newlab0}, \eqref{ellip} and \eqref{cond1} hold. 
\begin{enumerate}
\item[(i)] If $f=0$ in \eqref{eq101}, then \eqref{eq101} has a solution $U$ for every $\mathfrak{B}$ in the class $\mathfrak{BC}$. Moreover, 
$\Phi(x,U,\nabla U) \, U\in L^1(\Omega)$ and \eqref{sesim} holds with 
$u=v=U$. 
\item[(ii)] 
If \eqref{info} is satisfied, then   
\eqref{eq101} has at least a solution for every $f\in L^1(\Omega)$ and $\mathfrak{B}$ in the class $\mathfrak{BC}_+$.
\end{enumerate}
\end{theorem}

Theorem~\ref{nth2} is new even when transposed to isotropic Leray--Lions type operators $\mathcal A$ from $W_0^{1,p}(\Omega)$ into $W^{-1,p'}(\Omega)$.  
This is due to the introduction of $\mathfrak{B}$ in \eqref{eq101}, which adds extra difficulties. Were a solution of \eqref{eq101} to exist in $W_0^{1,\overrightarrow{p}}(\Omega)$, then we would expect it to be {\em unbounded}. This was observed by Bensoussan, Boccardo and Murat \cite{BBM} for isotropic nonlinear elliptic equations involving
$h\in W^{-1,p'}(\Omega)$ and Leray--Lions type operators of the $p$-Laplacian type. Hence, the regularizing effect that otherwise $\Phi$ would bring to the solutions in $W_0^{1,\overrightarrow{p}}(\Omega)$ is countered by 
the presence of $\mathfrak B$ in our class $\mathfrak{BC}$.

Note that without the term $\Phi$, one cannot expect to find  solutions of \eqref{eq101} in $W_0^{1,\overrightarrow{p}}(\Omega)$ for every $f\in L^1(\Omega)$. In the isotropic case, this observation was made by Boccardo and Gallou\"{e}t \cite{BGal}. 
A critical role in obtaining the existence of solutions of \eqref{eq101} in 
$W_0^{1,\overrightarrow{p}}(\Omega)$ 
is played by a gradient-dependent lower-order term $\Phi(x,u,\nabla u)$ with an ``anisotropic natural growth" in the gradient and a good sign condition (see \eqref{cond1} and \eqref{info}). 
 
For Theorem~\ref{nth2} (ii) we encounter two obstacles: a low summability for $f$ and, on the other hand, the unrestricted growth of $\Phi$ with respect to $|u|$. 
Previously mentioned works in the isotropic case provide ways to surmount one problem at a time. 
The function  
$f\in L^1(\Omega)$ can surely be approximated by $L^\infty(\Omega)$-functions $f_\varepsilon$ in the sense that $|f_\varepsilon|\leq |f| $ and $f_\varepsilon\to f$ a.e in $\Omega$ as $\varepsilon\to 0$. Also $\Phi$ could be replaced by a ``nice" function $\Phi_\varepsilon$,
preserving the properties of $\Phi$, but gaining boundedness, see \eqref{fina}. 
However, as it was pointed out by Bensoussan and 
Boccardo \cite{BBoc} in the isotropic case, one 
cannot deal with both approximations for $f$ and $\Phi$ {\em simultaneously}. 
For the approximate problems involving both $\Phi_\varepsilon$ and $f_\varepsilon$, we would not be able to obtain that the solutions $u_\varepsilon$ are uniformly bounded in $W_0^{1,\overrightarrow{p}}(\Omega)$ with respect to $\varepsilon$. 
For the above reason, we need to consider $f=0$ first and prove  
Theorem~\ref{nth2} (i), which is a crucial step in establishing the second assertion of Theorem~\ref{nth2}, but at the same time of independent interest. 
%In Theorem~\ref{nth2} (i), we don't need \eqref{info} and the operator $\mathfrak B$ can be chosen arbitrarily in the class $\mathfrak{BC}$.

%\begin{theorem} \label{nth} 
%Let $f=0$ in \eqref{eq101}. If \eqref{IntroEq0}, \eqref{newlab0},  \eqref{ellip}, and \eqref{cond1} hold, then \eqref{eq101} has a solution $U$ for every $\mathfrak{B}$ in the class $\mathfrak{BC}$. Moreover, 
%$\Phi(x,U,\nabla U) \, U\in L^1(\Omega)$ and \eqref{sesim} holds with 
%$u=v=U$. 
%\end{theorem}
% }

The techniques and results we obtain here provide the means to address other types of lower order terms than $f\in L^1(\Omega)$, yet maintaining the class $\mathfrak{BC}$ of operators $\mathfrak{B}$. 
We briefly mention possible developments. It is natural to ask what happens when $\mathfrak{b}_j$ in \eqref{exa} is negative for $1\leq j\leq N$.
Then, the sign-condition on $\Phi$ in \eqref{cond1} breaks down. 
Since we impose no restriction on the growth of $\Phi$ with respect to $|u|$, the current paper lays the foundation for dealing with 
potentially singular 
lower order terms with no sign restriction. 
The approximation of such terms is afforded by our inclusion in \eqref{eq101} of the term $\Theta$. The approximate problems become of the type \eqref{eq101} for which we gain existence of solutions via our Theorem~\ref{nth2} (i). It is essential that we can take the solution itself as a test function. This fact can be exploited to obtain {\em a priori} estimates for the solutions and pass to the limit. 
Such an analysis goes beyond the scope of this paper and will be carried out elsewhere (see \cite{BC}).    

Our work is also motivated by the various applications of anisotropic elliptic and parabolic partial differential equations 
to the mathematical modelling of physical and mechanical processes. Such equations  provide, for instance, the mathematical models for the dynamics of fluids in anisotropic media when the conductivities of the media are distinct in different directions (see \cite{ADS}). They also appear in biology as a model for the propagation of epidemic diseases in heterogeneous domains \cite{BK}. 
With a rapidly growing literature on anisotropic problems, several questions have been resolved on the existence, uniqueness and regularity of weak solutions (see, for instance, \cite{ACCZ,AdBF,AC,AS,BDP,BMS,C,CV,dBFZ,DM,FVV,FGK,FGL,GLR,M}).
Many difficulties arise in passing from the isotropic setting to the anisotropic one 
since some fundamental tools available for the former (such as the strong maximum principle, see \cite{V}) cannot be extended to the latter.   

\vspace{-0.2cm}
%anisotropic elliptic and parabolic partial differential equations represent a burgeoning area of research because of their applications 
\subsection{Our assumptions} \label{asb} 
Let \eqref{IntroEq0} and \eqref{newlab0} hold.   
The anisotropic $\overrightarrow{p}$-Laplacian in \eqref{lapal} is the prototype for a coercive, bounded, continuous and pseudo-monotone operator $\mathcal A:W_0^{1,\overrightarrow{p}}(\Omega)\to W^{-1,\overrightarrow{p}'}(\Omega)$ in divergence form  
$\mathcal{A} u=
-\sum_{j=1}^N \partial_j (A_j(x,u,\nabla u))$, that is, 
\begin{equation} \label{form}
\langle \mathcal A u,v\rangle =\sum_{j=1}^N \int_\Omega A_j(x,u,\nabla u)\,\partial_j v\,dx\quad \mbox{for every } u,v\in  W_0^{1,\overrightarrow{p}}(\Omega).
\end{equation}

\medskip
\noindent $\bullet$ 
For each $1\leq j\leq N$, let $A_j(x,t,\xi):\Omega\times \mathbb R\times \mathbb R^N\to  \mathbb R$ be a Carath\' eodory function and assume that there exist constants $\nu_0,\nu>0$ and a nonnegative function  
$\eta_j\in L^{p_j'}(\Omega)$
such that for a.e. $x\in \Omega$, for all $(t,\xi)\in  \mathbb R\times \mathbb R^N$ and every $\widehat \xi\in \mathbb R^N$, we have
\begin{equation} \label{ellip}
\left.\begin{aligned}
& \sum_{i=1}^N A_i(x,t,\xi) \,\xi_i\geq \nu_0\sum_{i=1}^N |\xi_i|^{p_i} && \mbox{[coercivity]},&\\
& \sum_{i=1}^N\left(A_i(x,t,\xi)- A_i(x,t,\widehat \xi) \right)\left(\xi_i-\widehat \xi_i\right)>0\quad \text{if } \xi\not=\widehat \xi && \mbox{[monotonicity]},&\\
& |A_j(x,t,\xi)|\leq \nu \left[ \eta_j(x)+|t|^{p^\ast/p_j'}+\left(\sum_{i=1}^N |\xi_i|^{p_i}\right)^{1/p_j'}
\right]&& \mbox{[growth condition]}.&
\end{aligned}\right\}
\end{equation}

We note that in the growth condition in \eqref{ellip}, we take the greatest exponent for $|t|$ from the viewpoint of the anisotropic Sobolev inequalities.
This requires modifying the standard proof of pseudo-monotonicity of $\mathcal A$   
(see Lemma~\ref{lem-tt11}).

\medskip

\noindent $\bullet$ 
Suppose that $\Phi(x,t,\xi):\Omega\times \mathbb R\times \mathbb R^N\to  \mathbb R$ is a Carath\' eodory function and there exist  
a nonnegative function $c\in L^1(\Omega)$ and a continuous nondecreasing function 
$ \phi :\mathbb R\to \mathbb R^+$ 
such that for a.e. $x\in \Omega$ and for all $(t,\xi)\in  \mathbb R\times \mathbb R^N$, 
\begin{equation} \label{cond1}
\Phi(x,t,\xi)\, t\geq 0 \ \ \mbox{[sign-condition]}, \ \  
|\Phi(x,t,\xi)|\leq \phi(|t|) \left( \sum_{j=1}^N |\xi_j|^{p_j} +c(x)\right).   
\end{equation} 

For Theorem~\ref{nth2} (ii), we further assume that 
there exist constants $\tau,\gamma>0$ such that 
\begin{equation} \label{info}
|\Phi(x,t,\xi)| \geq \gamma \sum_{j=1}^N |\xi_j|^{p_j} \quad \text{for all } |t|\geq \tau,\ \mbox{a.e. } x\in \Omega\ \mbox{and all } \xi\in \mathbb R^N.
\end{equation}

\subsection{Sketch of the main ideas in the proof of Theorem~\ref{nth2}} \label{mainres}
 
%When $f=0$ in \eqref{eq101}, we obtain a solution with better properties and under weaker assumptions than those in Theorem~\ref{nth2}. 
We remark that because of $\Phi$, even when $f=0$, we cannot directly apply the theory of pseudo-monotone operators to prove the existence claim in Theorem~\ref{nth2} (i). To overcome this difficulty, we consider 
the approximate problem  
\begin{equation} \label{love}
\left\{  \begin{aligned}
& \mathcal A u_\varepsilon + \Phi_\varepsilon(x,u_\varepsilon,\nabla u_\varepsilon) +\Theta(x,u_\varepsilon,\nabla u_\varepsilon)=\mathfrak{B}u_\varepsilon\quad \mbox{in }  \Omega,\\
& u_\varepsilon\in W_0^{1,\overrightarrow{p}}(\Omega)
\end{aligned}  \right. \end{equation}
for which we obtain the existence of a solution $u_\varepsilon$ as a consequence of our Theorem~\ref{het} in Section~\ref{newso}. Indeed, $\Phi_\varepsilon+\Theta$ satisfies the same type of assumption as $\Theta$ in \eqref{newlab0}, that is, there exists a constant $C_\varepsilon>0$ such that $|(\Phi_\varepsilon+\Theta)(x,t,\xi)|\leq C_\varepsilon$ for a.e. $x\in \Omega$ and all $(t,\xi)\in \mathbb R\times \mathbb R^N$. 
Thus, by Theorem~\ref{het}, for every $\varepsilon>0$, the approximate problem \eqref{love} has a solution $u_\varepsilon\in W_0^{1,\overrightarrow{p}}(\Omega)$. 
In Lemma~\ref{l1} we prove {\em a priori} estimates in $W_0^{1,\overrightarrow{p}}(\Omega)$ for the solutions $u_\varepsilon$, which (up to a subsequence) converge weakly to some $U$ in $W_0^{1,\overrightarrow{p}}(\Omega)$ and a.e. in $\Omega$ as $\varepsilon\to 0$.  

We point out that in Section~\ref{stil}, we will be able to show that, up to a subsequence, 
\begin{equation} 
\label{oil}
u_\varepsilon\to U\ \mbox{(strongly) in } W_0^{1,\overrightarrow{p}}(\Omega) \ \ \mbox{as } \varepsilon\to 0.
\end{equation}
We achieve this by combining and extending techniques from the isotropic case in \cite{BBoc} and \cite{BOG} to establish in Lemma~\ref{arc} that, up to a subsequence of $u_\varepsilon$, we have 
\begin{equation} \label{extram}  \nabla u_\varepsilon\to \nabla U\ \mbox{a.e. in } \Omega\ \mbox{and } 
T_k (u_\varepsilon)\to T_k (U)\ \mbox{(strongly) in } W_0^{1,\overrightarrow{p}}(\Omega)\ \mbox{as } \varepsilon\to 0
\end{equation} for every integer $k\geq 1$, where 
$T_k(\cdot)$ is given in \eqref{trunc}. 
Then, we can pass to the limit as $\varepsilon\to 0$ in the weak formulation of the solution $u_\varepsilon$ and  obtain that $U$ is a solution of \eqref{eq101} with $f=0$ (see Subsection~\ref{pass}). 
%In Section~\ref{stil}, we improve \eqref{extram} in the form of \eqref{oil}.  

Generally speaking, the proof of Theorem~\ref{nth2} (ii), which we give in Section~\ref{lastsec}, follows a similar course with that of Theorem~\ref{nth2} (i) in Section~\ref{sec6}. But there are some modifications that we outline below. We approximate $f\in L^1(\Omega)$ by $L^\infty(\Omega)$-functions $f_\varepsilon$ and we apply Theorem~\ref{nth2} (i) to obtain a solution $U_\varepsilon$ for the problem
\begin{equation} \label{love2}
\left\{ \begin{aligned}
& \mathcal A U_\varepsilon + \Phi(x,U_\varepsilon,\nabla U_\varepsilon) +\Theta(x,U_\varepsilon,\nabla U_\varepsilon)=\mathfrak{B}U_\varepsilon+f_\varepsilon\quad \mbox{in }  \Omega,\\
&  U_\varepsilon\in W_0^{1,\overrightarrow{p}}(\Omega),\quad 
\Phi(x,U_\varepsilon,\nabla U_\varepsilon)\in L^1(\Omega) 
.
\end{aligned}\right.
\end{equation}  
We emphasize that unlike in \eqref{love}, we have $\Phi$ (and not $\Phi_\varepsilon$) in \eqref{love2}. Because of this reason, coupled with the introduction of $f_\varepsilon$, we need the extra assumption \eqref{info} and to choose $\mathfrak{B}$ in the class $\mathfrak{BC}_+$ to obtain that $\{U_\varepsilon\}_\varepsilon$ is uniformly bounded in $W_0^{1,\overrightarrow{p}}(\Omega)$ with respect to $\varepsilon$
(see Lemma~\ref{lem-ad} for details). 
Then, extracting a subsequence, $U_\varepsilon$ tends to some $U_0$ weakly in $W_0^{1,\overrightarrow{p}}(\Omega)$ and a.e. in $\Omega$. With an almost identical argument, we gain the counterpart of \eqref{extram}, namely, up to a subsequence,
$ \nabla U_\varepsilon\to \nabla U_0$ a.e. in $\Omega$ and  
$T_k (U_\varepsilon)\to T_k (U_0)$ (strongly) in  $W_0^{1,\overrightarrow{p}}(\Omega)$ as $\varepsilon\to 0$ 
for every integer $k\geq 1$. To conclude the proof of Theorem~\ref{nth2} (ii), it remains to pass to the limit in the weak formulation of $U_\varepsilon$. The change appearing here compared with the corresponding argument in Subsection~\ref{pass} is the strong convergence of $\Phi(x,U_\varepsilon,\nabla U_\varepsilon)$ to $ \Phi(x,U_0,\nabla U_0)$ in $L^1(\Omega)$. For the latter, we adapt an argument from 
\cite{BOG}. For details, we refer to Lemma~\ref{coro2} in Subsection~\ref{limpas}.  

\medskip

\noindent {\bf Structure of this paper.} 
In Section~\ref{newso} we prove an existence result (Theorem~\ref{het}), which gives the existence of a solution $u_\varepsilon$ of \eqref{love} for every $\varepsilon>0$. We dedicate Sections~\ref{sec6} and \ref{lastsec}  to the proof of Theorem~\ref{nth2} (i) and Theorem~\ref{nth2} (ii), respectively. In Section~\ref{stil} we make further comments on Theorem~\ref{nth2} (i) by proving the strong convergence in \eqref{oil}. In the Appendix we include some facts used in the paper and, for completeness, prove 
the anisotropic counterparts of well-known isotropic convergence results, see
Lemmas \ref{gnsb1} and \ref{joc1}. These will be used in 
the proof of Lemmas~\ref{lem-tt11} and \ref{arc}, respectively.

\medskip
\noindent {\bf Notation.} For $k>0$, we let $T_k:\mathbb R\to \mathbb R$ stand for the truncation at height $k$, that is,  
\begin{equation} \label{trunc} T_k(s)=s \quad \mbox{if } |s| \le k, \quad T_k(s)=k\, \frac{s}{|s|} \quad \mbox{if }  |s|>k.\end{equation}
\noindent Moreover, we define $G_k:\mathbb R\to \mathbb R$ by  
\begin{equation}\label{gk}
G_k(s)=s-T_k(s)\quad \mbox{for every }s\in \R.
\end{equation}
In particular, we have $G_k=0$ on $[-k,k]$ and $t\,G_k(t)\geq 0$ for every $t\in \mathbb R$.

\noindent For every $u\in  W_0^{1,\overrightarrow{p}}(\Omega)$ and 
for a.e. $x\in \Omega$, we define  
\begin{equation*} \label{zeta}  
\begin{aligned}
& \widehat{\Phi}(u)(x):=\Phi(x,u(x),\nabla u(x)),\quad  \widehat{\Theta}(u)(x):=\Theta(x,u(x),\nabla u(x)), \\
& \widehat{A}_j (u)(x)=A_j(x,u(x),\nabla u(x))\quad \text{for every } 1\leq j\leq N.
\end{aligned}
\end{equation*}

We set $\overrightarrow{p}=\left(p_1,p_2,\ldots,p_N\right)$ and 
$\overrightarrow{p}'=(p_1',p_2',\ldots,p_N')$. 

As usual, $\chi_\omega$ stands for the characteristic function of a set $\omega\subset \mathbb R^N$.

\section{An existence result} \label{newso}

Throughout this section, we assume \eqref{IntroEq0}, \eqref{newlab0}, and \eqref{ellip}, besides   
$\mathfrak{B}$ belonging to the class $\mathfrak{BC}$. Here, our aim is to prove the existence of a solution to the following problem
\begin{equation} \label{ditp}
\left\{  \begin{aligned} 
&\mathcal A u+\Theta (x,u,\nabla u)=\mathfrak{B}u \quad \mbox{in }  \Omega,\\ 
& u\in W_0^{1,\overrightarrow{p}}(\Omega).
\end{aligned} \right. 
\end{equation}
By a solution of \eqref{ditp}, we mean 
a function 
$u\in  W_0^{1,\overrightarrow{p}}(\Omega)$ such that 
\begin{equation} \label{ssdil}
\langle \mathcal Au,v\rangle
+
\int_\Omega \Theta(x,u,\nabla u)\,v\,dx-\langle \mathfrak{B} u,v\rangle =0 \quad \mbox{for every }v\in W_0^{1,\overrightarrow{p}}(\Omega) . 
\end{equation}	

\begin{theorem} \label{het} 
Problem \eqref{ditp} admits at least a solution.
\end{theorem}
We establish Theorem~\ref{het} via the theory of pseudo-monotone operators. Before giving the proof of Theorem~\ref{het} in Subsection~\ref{subn},  
we recall a few concepts that we need in the sequel (see, for example, \cite{bre} and \cite[p. 586]{ze}). 

\begin{defi} \label{nsr}
	{\rm An operator $\mathcal P:W_0^{1,\overrightarrow{p}}(\Omega)\to W^{-1,\overrightarrow{p}'}(\Omega)$  
		is called 
		
		($a_1$) {\em monotone} {\em (strictly monotone)} if $\langle \mathcal P u-\mathcal Pv,u-v\rangle\geq 0$ for every $u,v\in W_0^{1,\overrightarrow{p}}(\Omega)$ (with equality if and only if $u=v$);
		
		($a_2$) 
		{\em pseudo-monotone} if whenever $u_\ell\rightharpoonup u$ (weakly) in $W_0^{1,\overrightarrow{p}}(\Omega)$ as $\ell\to \infty$ and 
		$\limsup_{\ell\to \infty} \langle \mathcal P u_\ell,u_\ell-u\rangle\leq 0$, we get that
		$ \langle  \mathcal Pu,u-w\rangle\leq 
		\liminf_{\ell\to \infty}   \langle  \mathcal P u_\ell,u_\ell-w\rangle$ for all $ w\in W_0^{1,\overrightarrow{p}}(\Omega)$;
		
		($a_3$) {\em strongly continuous}\footnote{\, Strongly continuous operators are also referred to as {\em completely continuous} (see, for instance, Showalter \cite[p. 36]{Sh}).} if $u_\ell\rightharpoonup u$ (weakly) in $W_0^{1,\overrightarrow{p}}(\Omega)$ as $\ell\to \infty$  implies that $\mathcal P u_\ell\to \mathcal Pu$ in $W^{-1,\overrightarrow{p}'}(\Omega)$ as $\ell\to \infty$. 
		
		($a_4$) {\em coercive} if $ \langle \mathcal P  u,u\rangle/\|u\|_{W_0^{1,\overrightarrow{p}}(\Omega)}\to \infty$ as $  \|u\|_{W_0^{1,\overrightarrow{p}}(\Omega)}\to \infty$;
		
		($a_5$) {\em of M type}\footnote{\, Some authors (see, for example, Le Dret \cite[p. 232]{L}) use the terminology {\em sense 1 pseudomonotone} instead of M type.} if $u_\ell\rightharpoonup u$ (weakly) in $W_0^{1,\overrightarrow{p}}(\Omega)$ as $\ell\to \infty$, together with  $\mathcal P u_\ell \rightharpoonup g$ (weakly) in $W^{-1,\overrightarrow{p}'}(\Omega)$ as $\ell\to \infty$ and  $\limsup_{\ell \to \infty} \langle \mathcal P u_\ell,u_\ell \rangle \le \langle g,u\rangle$, imply that 
		$ g=\mathcal P u$ and $\langle \mathcal P u_\ell,u_\ell\rangle \to \langle g,u\rangle\ \mbox{as } \ell\to \infty.
		$	
	} 
\end{defi}	

\begin{proposition}\label{prop_pmm}
	Every strongly continuous operator $\mathcal P:W_0^{1,\overrightarrow{p}}(\Omega)\to W^{-1,\overrightarrow{p}'}(\Omega)$ is pseudo-monotone.  Every bounded operator $\mathcal P:W_0^{1,\overrightarrow{p}}(\Omega)\to W^{-1,\overrightarrow{p}'}(\Omega)$ of M type is pseudo-monotone. The sum of two pseudo-monotone operators is pseudo-monotone.
\end{proposition}

\subsection{Proof of Theorem~\ref{het}} \label{subn}
We immediately observe from \eqref{newlab0} that the operator $\mathcal P_\Theta:W_0^{1,\overrightarrow{p}}(\Omega)\to W^{-1,\overrightarrow{p}'}(\Omega)$ is bounded, where we define
\begin{equation} \label{pobn}\langle  \mathcal P_\Theta(u),v\rangle:=\int_\Omega  \widehat{\Theta}(u)\,v\,dx\quad \mbox{for every }u,v\in W_0^{1,\overrightarrow{p}}(\Omega). 
\end{equation} 

In view of \eqref{ssdil},   
the existence of a solution to \eqref{ditp} follows whenever the operator $\mathcal A+\mathcal P_\Theta-\mathfrak{B}:W_0^{1,\overrightarrow{p}}(\Omega)\to W^{-1,\overrightarrow{p}'}(\Omega)$ is surjective. 
Since $W_0^{1,\overrightarrow{p}}(\Omega)$ is a 
real, reflexive, and separable Banach space, it is known that $\mathcal A+\mathcal P_\Theta-\mathfrak{B}:W_0^{1,\overrightarrow{p}}(\Omega)\to W^{-1,\overrightarrow{p}'}(\Omega)$ is surjective whenever it is bounded, coercive and pseudo-monotone (see, for instance, \cite[p. 589]{ze}). In Lemma~\ref{ale0}, we establish the boundedness and coercivity of $\mathcal A+\mathcal P_\Theta-\mathfrak{B}$, whereas its pseudo-monotonicity is concluded in Corollary~\ref{pso}. 

For the reader's convenience and to make our presentation self-contained, we give all the details about the pseudo-monotonicity of $\mathcal A+\mathcal P_\Theta-\mathfrak{B}:W_0^{1,\overrightarrow{p}}(\Omega)\to W^{-1,\overrightarrow{p}'}(\Omega)$. These computations could be of interest also in the corresponding isotropic case, when, to our best knowledge,  only very special instances of $\mathfrak{B}$ have been considered and the details are usually scattered in the literature.

The property $(P_2)$  
ensures that $\pm\mathfrak{B}:W_0^{1,\overrightarrow{p}}(\Omega)\to W^{-1,\overrightarrow{p}'}(\Omega)$ is strongly continuous (see Lemma~\ref{strong0}) and, hence, pseudo-monotone by Proposition~\ref{prop_pmm}. 
As the sum of two pseudo-monotone operators is pseudo-monotone, to prove that $\mathcal A+\mathcal P_\Theta-\mathfrak{B}:W_0^{1,\overrightarrow{p}}(\Omega)\to W^{-1,\overrightarrow{p}'}(\Omega)$ is pseudo-monotone, it suffices to show that $\mathcal A+\mathcal P_\Theta:W_0^{1,\overrightarrow{p}}(\Omega)\to W^{-1,\overrightarrow{p}'}(\Omega)$
is pseudo-monotone. The proof of the latter is more involved, see   
Lemma~\ref{lem-tt11}. In view of Proposition~\ref{prop_pmm} and Lemma~\ref{lem-tt00}, it is enough to show that $\mathcal A+\mathcal P_\Theta$ is an operator of M type.   
We proceed with the details.

\begin{lemma} \label{lem-tt00} 	
The operator $\mathcal A +\mathcal P_\Theta  
:W_0^{1,\overrightarrow{p}}(\Omega)\to W^{-1,\overrightarrow{p}'}(\Omega)$ is  bounded, coercive and continuous. 
\end{lemma}
\begin{proof}
The boundedness of the operator   
$\mathcal A +\mathcal P_\Theta :W_0^{1,\overrightarrow{p}}(\Omega)\to W^{-1,\overrightarrow{p}'}(\Omega)$ is a consequence of the growth condition of $ A_j$ in \eqref{ellip}, coupled with \eqref{newlab0}.   
The coercivity of $\mathcal A +\mathcal P_\Theta$ follows readily from \eqref{newlab0} and the coercivity assumption in \eqref{ellip}. 
Moreover, by H\"older's inequality and the continuity of the embedding $W_0^{1,\overrightarrow{p}}(\Omega) \hookrightarrow L^{p^\ast}(\Omega)$, we find a positive constant $C$ such that, for every $u_1,u_2\in W_0^{1,\overrightarrow{p}}(\Omega)$, 
$$
\begin{aligned} &\|(\mathcal A + \mathcal P_\Theta
) (u_1)- (\mathcal A+ \mathcal P_\Theta) (u_2)\|_{W^{-1,\overrightarrow{p}'}(\Omega)}
\\
& \leq \sup_{
	\substack{
		v\in W_0^{1,\overrightarrow{p}}(\Omega),\\  \|v\|_{W_0^{1,\overrightarrow{p}}(\Omega)}\leq 1}}
\Big(\sum_{j=1}^N \int_\Omega |\widehat{A}_j(u_1)-\widehat{A}_j(u_2)| |\partial_j v|\,dx+\int_\Omega |
\widehat{\Theta}
(u_1)- \widehat{\Theta} (u_2)| |v| \,dx \Big)\\
&\leq \sum_{j=1}^N \|\widehat{A}_j(u_1)-\widehat{A}_j(u_2)\|_{L^{p_j'}(\Omega)}+C\, ||   \widehat{\Theta} (u_1)-  \widehat{\Theta}(u_2)||_{L^{(p^\ast)'}(\Omega)}. 
\end{aligned}
$$

We get the continuity of $\mathcal A+\mathcal P_\Theta
:W_0^{1,\overrightarrow{p}}(\Omega)\to W^{-1,\overrightarrow{p}'}(\Omega)$ by showing the following.

\vspace{0.2cm}
\noindent {\bf Claim:} The mappings $\widehat{\Theta}:W_0^{1,\overrightarrow{p}}(\Omega)\to L^{(p^\ast)'}(\Omega)$ and $\widehat{A}_j:W_0^{1,\overrightarrow{p}}(\Omega)\to L^{p_j'}(\Omega)$ are continuous for each $1\leq j\leq N$.

\vspace{0.2cm}
\noindent {\em Proof of the Claim.}
	Let $1\leq j\leq N$ be arbitrary.   
	By the growth condition of $A_j$ in \eqref{ellip}, there exist a constant $C>0 $ and a nonnegative function $\eta_j\in L^{p_j'}(\Omega)$ such that 
	\begin{equation}\label{pm50}
	|\widehat{A}_j(u)|^{p_j'}\leq C \left( \eta_j^{p_j'}+|u|^{p^\ast} +
	\sum_{i=1}^N |\partial_i u|^{p_i}	\right)\in L^1(\Omega)
	\end{equation} 
		for all $ u\in W_0^{1,\overrightarrow{p}}(\Omega)$. 
	Since the embeddings $W_0^{1,\overrightarrow{p}}(\Omega)\hookrightarrow L^{p^\ast}(\Omega)$ and $L^\infty(\Omega)\hookrightarrow L^{(p^\ast)'}(\Omega)$ are continuous, from \eqref{pm50} and \eqref{newlab0}, we infer that 
	$\widehat{A}_j:W_0^{1,\overrightarrow{p}}(\Omega)\to L^{p_j'}(\Omega)$ and $\widehat{\Theta}:W_0^{1,\overrightarrow{p}}(\Omega)\to L^{(p^\ast)'}(\Omega)$ are well-defined. 
	To prove the continuity of these mappings, we 
	let $u_n\to u$ (strongly) in $W_0^{1,\overrightarrow{p}}(\Omega)$ as $n\to \infty$.
	Hence, $u_n\to u$ (strongly) in $L^{p^\ast}(\Omega)$ and $\partial_i u_{n}\to \partial_i u$ (strongly) in $L^{p_i}(\Omega)$ as $n\to \infty$ for every $1\leq i\leq N$. Now, using \eqref{pm50} with $u_n$ instead of $u$, we obtain that $\{|\widehat{A}_j(u_n)|^{p_j'}\}_{n\geq 1}$ is uniformly integrable over $\Omega$. 
	By passing to    
	a subsequence $\{u_{n_k}\}_{k\geq 1}$ of $\{u_n\}$, we have
	$u_{n_k}\to u$ and $\nabla u_{n_k}\to \nabla u$ a.e. in $\Omega$ as $k\to \infty$. Since $A_j$ and $\Theta$ are  Carath\' eodory functions, we have $\widehat{\Theta}(u_{n_k})\to \widehat{\Theta}(u)$ and 
	$\widehat{A}_j(u_{n_k})\to \widehat{A}_j(u)$ a.e. in $\Omega$ as $k\to \infty$.  
	Then, by \eqref{newlab0} and the Dominated Convergence Theorem, 
	$\widehat{\Theta}(u_{n_k})\to \widehat{\Theta}(u)$ in $L^{(p^\ast)'}(\Omega)$. By Vitali's Theorem, we see that  
	$\widehat{A}_j(u_{n_k})\to \widehat{A}_j(u)$  in $L^{p_j'}(\Omega)$
	as $k\to \infty$. 
	Since the limits $\widehat\Theta(u)$ and $\widehat{A}_j(u)$ are independent of the subsequence $\{u_{n_k}\}_{k\geq 1}$, we conclude that $\widehat{\Theta}(u_{n})\to \widehat{\Theta}(u)$ in $L^{(p^\ast)'}(\Omega)$ and $\widehat{A}_j(u_{n})\to \widehat{A}_j(u)$ in $L^{p_j'}(\Omega)$ as $n\to \infty$. 

	This completes the proof of the Claim and of Lemma~\ref{lem-tt00}. 
\end{proof}

\begin{lemma} \label{ale0} 
	The operator $\mathcal A+\mathcal P_\Theta-\mathfrak{B}:W_0^{1,\overrightarrow{p}}(\Omega) \to W^{-1,\overrightarrow{p}'}(\Omega)$ is bounded and coercive. 
\end{lemma}

\begin{proof}
	Using Lemma \ref{lem-tt00} and Definition~\ref{doi}, we find that $\mathcal A+\mathcal P_\Theta-\mathfrak{B}$
	is a bounded operator from
	$ W_0^{1,\overrightarrow{p}}(\Omega)$ into $ W^{-1,\overrightarrow{p}'}(\Omega)$. 
	We now show that it is also coercive, namely,  
	\begin{equation} \label{coer0} \frac{ \langle \mathcal Au+\mathcal P_\Theta
		(u)-\mathfrak{B} u,u\rangle}{\|u\|_{W_0^{1,\overrightarrow{p}}(\Omega)}}\to \infty
	\ \mbox{ as }\  
	\|u\|_{W_0^{1,\overrightarrow{p}}(\Omega)}\to \infty.\end{equation}
	Using \eqref{newlab0} and the continuity of the embedding $W_0^{1,\overrightarrow{p}}(\Omega)\hookrightarrow L^{1}(\Omega)$, we find a constant $C>0$ such that 
	$ \langle \mathcal P_\Theta
	(u) ,u\rangle\geq - C \|u\|_{W_0^{1,\overrightarrow{p}}(\Omega)}$ 
	for every $ u\in W_0^{1,\overrightarrow{p}}(\Omega)$. Then, by the coercivity property of $\mathcal A-\mathfrak B$, 
	we readily conclude \eqref{coer0}. 
	\end{proof}

\begin{lemma} \label{strong0} 
	Every operator $\mathcal{B}:W_0^{1,\overrightarrow{p}}(\Omega)\to W^{-1,\overrightarrow{p}'}(\Omega)$ satisfying $(P_2)$ is strongly continuous.   
\end{lemma}
\begin{proof}
	Let $u_\ell\rightharpoonup u$ (weakly) in $W_0^{1,\overrightarrow{p}}(\Omega)$ as $\ell\to \infty$. We show that $\mathcal{B} u_\ell\to \mathcal{B} u$ in $W^{-1,\overrightarrow{p}'}(\Omega)$ as $\ell\to \infty$.  	
	Assume by contradiction that there exist $\varepsilon_0>0$ and a subsequence of $\{u_\ell\}$ (relabeled $\{u_\ell\}$) such that 
	$$ \sup_{\substack{
			v\in W_0^{1,\overrightarrow{p}}(\Omega),\\  \|v\|_{W_0^{1,\overrightarrow{p}}(\Omega)}\leq 1}} 
	\left|\langle \mathcal{B} u_\ell -\mathcal{B} u,v\rangle 
	\right|  >\varepsilon_0\quad \mbox{for every } \ell\geq 1. 
	$$ Hence, there also exists $\{v_\ell\}$ in $W_0^{1,\overrightarrow{p}}(\Omega)$ with $\|v_\ell\|_{W_0^{1,\overrightarrow{p}}(\Omega)}\leq 1$ such that 
	\begin{equation}\label{cont0}  \left|\langle \mathcal{B} u_\ell -\mathcal{B} u,v_\ell\rangle 
	\right|>\varepsilon_0\quad \mbox{for all } \ell\geq 1.
	\end{equation}
	By the boundedness of $\{v_\ell\}$ in $W_0^{1,\overrightarrow{p}}(\Omega)$, up to a subsequence,  
	$v_\ell\rightharpoonup v$ (weakly) in 
	$W_0^{1,\overrightarrow{p}}(\Omega)$ as $\ell\to \infty$.
	Since $\mathcal{B} u\in W^{-1,\overrightarrow{p}'}(\Omega) $, we have 
	$ \langle \mathcal{B} u,v_\ell \rangle\to  \langle \mathcal{B} u,v \rangle \ \mbox{as } \ell\to \infty.
	$
	Hence, from ($P_2$)  we find that $ \left|\langle \mathcal{B} u_\ell -\mathcal{B} u,v_\ell\rangle 
	\right| \to 0$ as $\ell \to \infty$, which is in contradiction with \eqref{cont0}. Thus, $\mathcal B$ is strongly continuous, completing the proof. \end{proof}

\begin{lemma} \label{lem-tt11}  
The operator $\mathcal A +\mathcal P_\Theta  
:W_0^{1,\overrightarrow{p}}(\Omega)\to W^{-1,\overrightarrow{p}'}(\Omega)$ is  pseudo-monotone. 
\end{lemma}
\begin{proof}
Since  the operator $\mathcal A+ \mathcal P_\Theta$ is bounded, it is enough to show that it is of M type (see Proposition \ref{prop_pmm}). To this end, suppose that there exist $u$, $\{u_\ell\}_{\ell \ge 1}$ in $ W_0^{1,\overrightarrow{p}}(\Omega)$ and $g \in W^{-1,\overrightarrow{p}'}(\Omega)$ such that
\begin{eqnarray} 
&&u_\ell \rightharpoonup u  \ \mbox{(weakly) in } W_0^{1,\overrightarrow{p}}(\Omega) \ \mathrm{as}\ \ell \to \infty, \label{pm10}
\\
&& (\mathcal A + \mathcal P_\Theta)(u_\ell) \rightharpoonup g \mbox{ (weakly) in } W^{-1,\overrightarrow{p}'}(\Omega) \ \mathrm{as}\  \ell \to \infty, \label{pmkl0}
\\
&&\displaystyle \limsup_{\ell \to \infty}\, \langle (\mathcal A + \mathcal P_\Theta)(u_\ell),u_\ell\rangle \le \langle g,u\rangle. \label{pm20}
\end{eqnarray}
We prove that
\begin{eqnarray}
&& g=(\mathcal A+ \mathcal P_\Theta)(u), \label{pm30}
\\
&& \langle (\mathcal A + \mathcal P_\Theta)(u_\ell),u_\ell\rangle \to \langle g,u\rangle \quad \mbox{as} \> \ell \to \infty. \label{pm40}
\end{eqnarray}
We first show that \eqref{pm40} holds. 
From \eqref{pm10} and the compactness of the embedding $W_0^{1,\overrightarrow{p}}(\Omega)
\hookrightarrow L^p(\Omega)$, we obtain that, up to a subsequence, 
\begin{equation} \label{strobe}  u_\ell \to u\ \mbox{strongly in } L^p(\Omega)\ \mbox{and a.e. in } \Omega.	
\end{equation}
Moreover, using \eqref{pm50} with $u$ replaced by $u_\ell$, we get that $\widehat{A}_j(u_\ell)$ is bounded in $L^{p_j'}(\Omega)$ for every $1 \le j \le N$. Hence, in view of \eqref{newlab0}, there exist $\mu\in L^{p'}(\Omega)$ and $g_j\in L^{p_j'}(\Omega)$ for $1\leq j\leq N$ so that, up to a further subsequence of $\{u_\ell\}$ (denoted by $\{u_\ell\}$), we have
\begin{equation}\label{pm60} 
\widehat \Theta
(u_\ell) \rightharpoonup  \mu \quad \mbox{(weakly) in } L^{p'}(\Omega)  \quad \mbox{and }\quad 
\widehat{A}_j (u_\ell) \rightharpoonup g_j \quad \mbox{(weakly) in } L^{p_j'}(\Omega) 
\end{equation}
as $\ell \to \infty$ 
for every $1\leq j\leq N$. Thus, by the reflexivity of $W_0^{1,\overrightarrow{p}}(\Omega)$ and \eqref{pmkl0}, we get 
\begin{equation} \label{mkslu0}
\langle g,v\rangle= \lim_{\ell\to \infty} 
\langle	(\mathcal A +\mathcal P_\Theta)(u_\ell),v\rangle
=\sum_{j=1}^N \int_\Omega g_j\,\partial_j v\,dx+\int_\Omega \mu \,v\,dx
\end{equation}
for every $v\in W_0^{1,\overrightarrow{p}}(\Omega)$. 
From \eqref{strobe} and \eqref{pm60}, we infer that
\begin{equation} \label{funk}
\lim_{\ell\to \infty} \int_\Omega \widehat{\Theta}(u_\ell)\,u_\ell\,dx=\int_\Omega \mu\,u\,dx.
\end{equation}
From \eqref{pm20}, \eqref{mkslu0} and \eqref{funk}, we obtain that
\begin{equation} \label{nsms0} 
\begin{aligned}
\limsup_{\ell \to \infty}\, \langle (\mathcal A + 
\mathcal P_\Theta )(u_\ell),u_\ell\rangle&=
\limsup_{\ell \to \infty} \left(\sum_{j=1}^N \int_\Omega \widehat{A}_j(u_\ell)\,\partial_j u_\ell\, dx+\int_\Omega 
\widehat{\Theta}(u_\ell)\,
u_\ell\,dx\right)\\
& \leq  \langle g, u \rangle = \sum_{j=1}^N \int_\Omega g_j \, \partial_j u \, dx+\int_\Omega \mu\,  u\, dx,
\end{aligned}
\end{equation}
that is, 
\begin{equation}\label{pm210}
\limsup_{\ell \to \infty} \sum_{j=1}^N \int_\Omega \widehat{A}_j(u_\ell)\,\partial_j u_\ell\, dx \le \sum_{j=1}^N \int_\Omega g_j 
\, \partial_j u
\, dx.
\end{equation}
In light of \eqref{funk}--\eqref{pm210}, we conclude \eqref{pm40} by showing that
\begin{equation}\label{pm220}
\liminf_{\ell \to \infty} \sum_{j=1}^N \int_\Omega \widehat{A}_j(u_\ell)\,\partial_j u_\ell\, dx \ge \sum_{j=1}^N \int_\Omega g_j\, \partial_j u \, dx.
\end{equation}

\noindent The proof of \eqref{pm220} is a bit different from the classical one in the isotropic case since in our growth condition on $A_j$ in \eqref{ellip}, we have taken the greatest exponent for $|t|$ from the viewpoint of the anisotropic Sobolev inequalities. Let us emphasize what is new compared with the classical proof. Let $1\leq j\leq N$ be arbitrary. 
Since $u_\ell \to u$ a.e. in $\Omega$ and $A_j$ is a Carath\'eodory function, 
we see that 
\begin{equation} \label{pointw}  A_j(x,u_\ell,\nabla u)\to A_j(x,u,\nabla u)\quad \mbox{a.e. in } \Omega. \end{equation}
The growth condition in \eqref{ellip} gives a constant $C>0$ and a nonnegative function $\eta_j\in L^{p_j'}(\Omega)$ such that
\begin{equation} \label{guy}  |A_j(x,u_\ell,\nabla u)|^{p_j'}\leq C \left( \eta_j^{p_j'}+|u_\ell|^{p^\ast} +
\sum_{i=1}^N |\partial_i u|^{p_i}\right) 
\end{equation}  for every $\ell\geq 1$.   
Because the power of $|u_\ell|$ in the right-hand side of \eqref{guy} is $p^\ast$, the critical exponent, 
the compactness of the embedding $W_0^{1,\overrightarrow{p}}(\Omega)
\hookrightarrow L^{p^\ast}(\Omega) $ fails, in general. 
Hence, 
we cannot claim anymore that $\{|A_j(x,u_\ell,\nabla u)|^{p_j'}\}_{\ell \ge 1}$ is uniformly integrable over $\Omega$. Thus, we cannot apply Vitali's theorem to deduce the strong convergence of 
$A_j(x,u_\ell,\nabla u) $ to $A_j(x,u,\nabla u)$  in $  L^{p_j'}(\Omega)$ as $ \ell\to \infty$. However, if we fix $k\geq 1$, then by the growth condition in \eqref{ellip}, we infer that  
$$ \{|A_j(x,u_\ell,\nabla u)|^{p_j'}\,\chi_{\{|u_\ell|\leq k\}} \}_{\ell\geq 1} \quad\mbox{is uniformly integrable over }\Omega.$$
Then, since $\chi_{\{|u_\ell|\leq k\}}\to \chi_{\{|u|\leq k\}}$ as $\ell\to \infty$, from  \eqref{pointw} and Vitali's theorem, we get 
\begin{equation} \label{mxn0}   A_j(x,u_\ell,\nabla u)\, \chi_{\{|u_\ell|\leq k\}} \to A_j(x,u,\nabla u)\, \chi_{\{|u|\leq k\}}  \mbox{ strongly in }  L^{p_j'}(\Omega)\ \mbox{as } \ell\to \infty.
\end{equation}
We return to the proof of \eqref{pm220} with modifications suggested by \eqref{mxn0}. By the Dominated Convergence Theorem, we obtain 
\eqref{pm220} by showing that for every integer $k\geq 1$, 
\begin{equation} \label{boi} 
\liminf_{\ell \to \infty} \sum_{j=1}^N \int_\Omega \widehat{A}_j(u_\ell)\,\partial_j u_\ell\, dx \ge \sum_{j=1}^N \int_\Omega g_j\, (\partial_j u) \, \chi_{ \{|u|\leq k \}} \, dx.
\end{equation}
%Indeed, by letting $k\to \infty$ in \eqref{boi} and applying the  we arrive at \eqref{pm220}.

\vspace{0.2cm}
\noindent  {\em Proof of \eqref{boi}}. Fix an integer $k\geq 1$. 
The coercivity condition in \eqref{ellip} yields that
\begin{equation}\label{winn}
\sum_{j=1}^N  \widehat{A}_j(u_\ell)\,\partial_j u_\ell  
\geq 
\sum_{j=1}^N \widehat{A}_j(u_\ell) 
\,(\partial_j u_\ell)
\,\chi_{\{|u_\ell|\leq k\}}.
\end{equation}
For the right-hand side of \eqref{winn}, we use  
the monotonicity condition in \eqref{ellip}, that is, 
\begin{equation} \label{limu}	
\begin{aligned}
\sum_{j=1}^N \widehat{A}_j(u_\ell) 
\,(\partial_j u_\ell)
\,\chi_{\{|u_\ell|\leq k\}}
\geq & \sum_{j=1}^N \widehat{A}_j(u_\ell)\,(\partial_j u)
\,\chi_{\{|u_\ell|\leq k\}}\\
&
+\sum_{j=1}^N A_j(x,u_\ell,\nabla u)\,
(\partial_j u_\ell-\partial_j u)\,\chi_{\{|u_\ell|\leq k\}}.
\end{aligned}	
\end{equation}
Let $1\leq j\leq N$ be arbitrary. By the Dominated Convergence Theorem, we have $(\partial_j u)
\,\chi_{\{|u_\ell|\leq k\}}\to (\partial_j u)
\,\chi_{\{|u|\leq k\}}$	strongly in $L^{p_j}(\Omega)$ as $\ell\to \infty$. Recall from \eqref{pm60} that $\widehat{A}_j(u_\ell) \rightharpoonup  g_j$ (weakly) in $L^{p_j'}(\Omega)$ as $\ell\to \infty$. Hence we have   
\begin{equation}\label{kim} 
\widehat{A}_j(u_\ell)\,(\partial_j u)\,\chi_{\{|u_\ell|\leq k\}}\to  g_j\,(\partial_j u) \chi_{\{|u|\leq k\}}\ \mbox{strongly in } L^1(\Omega) \ \mbox{as } \ell\to \infty.
\end{equation} 
Since $\partial_j u_\ell \rightharpoonup \partial_j u$ (weakly) in $L^{p_j}(\Omega)$ as $\ell\to \infty$, using \eqref{mxn0}, we gain the following
\begin{equation}\label{pm80}
A_j(x,u_\ell,\nabla u)\,	(\partial_j u_\ell-\partial_j u)\,\chi_{\{|u_\ell|\leq k\}}
\to 0\quad \mbox{strongly in } L^1(\Omega). 
\end{equation}
In light of \eqref{kim} and \eqref{pm80}, we see 
that
$$ \sum_{j=1}^N \int_\Omega \widehat{A}_j(u_\ell)\,(\partial_j u)
\,\chi_{\{|u_\ell|\leq k\}}
+\sum_{j=1}^N \int_\Omega A_j(x,u_\ell,\nabla u)\,
(\partial_j u_\ell-\partial_j u)\,\chi_{\{|u_\ell|\leq k\}} 
$$ 
converges as $\ell\to \infty$ to the right-hand side of \eqref{boi}. Using this convergence, jointly with the inequalities in \eqref{winn} and \eqref{limu}, we conclude the proof of \eqref{boi}. 

\vspace{0.2cm}
As mentioned above, from \eqref{boi} we obtain \eqref{pm220}.  
Inequalities \eqref{pm210} and \eqref{pm220} ensure that
\begin{equation} \label{dgsb}
\lim_{\ell \to \infty} \sum_{j=1}^N \int_\Omega \widehat{A}_j(u_\ell)\,\partial_j u_\ell \, dx =\sum_{j=1}^N \int_\Omega g_j \,\partial_j u \, dx.
\end{equation}

It remains to establish \eqref{pm30}. 
From  \eqref{kim}--\eqref{dgsb}, we get 
\begin{equation} \label{emi} \sum_{j=1}^N \int_\Omega \left[
A_j(x,u_\ell,\nabla u_\ell)-A_j(x,u_\ell,\nabla u)
\right] (\partial_j u_\ell-\partial_j u) \,\chi_{\{|u_\ell|\leq k\}}\, dx\to 0\quad \mbox{as } \ell\to \infty. 
\end{equation} 
By \eqref{emi} and the monotonicity condition in \eqref{ellip}, we infer that 
$$ \sum_{j=1}^N \left[
A_j(x,u_\ell,\nabla u_\ell)-A_j(x,u_\ell,\nabla u)
\right] (\partial_j u_\ell-\partial_j u) \to 0\ \mbox{a.e in } \{|u_\ell|\leq k\}  \ \mbox{as }\ell\to \infty.
$$
By a standard diagonal argument, we can find a subsequence of $\{u_\ell\}$ (still denoted by $\{u_\ell\}$) such that the above convergence holds for every $k\geq 1$. This implies that 
$$  \sum_{j=1}^N \left[
A_j(x,u_\ell,\nabla u_\ell)-A_j(x,u_\ell,\nabla u)
\right] (\partial_j u_\ell-\partial_j u) \to 0\ \mbox{a.e. in }\Omega  \ \mbox{as }\ell\to \infty.
$$
In the notation of Subsection~\ref{prel} in the Appendix, we have
$ \mathcal D_{u_\ell}(u_\ell,u)\to 0$ a.e. in $ \Omega$ as  $\ell\to \infty.  
$
Thus, by Lemma~\ref{gnsb1} in the Appendix, up to a subsequence, 
$
\nabla u_\ell\to \nabla u$ a.e. in $\Omega$ as $\ell\to \infty$.  
Since $\Phi$ and $A_j$ (with $1\leq j\leq N$) are Carath\'eodory 
functions,  we find that 
$ \widehat{\Theta}(u_\ell)\to \widehat{\Theta}(u)$ and $ 
\widehat{A}_j(u_\ell)\to \widehat{A}_j(u)$ a.e. in $\Omega$ as $\ell\to \infty.$ Using this fact, jointly with \eqref{pm60}, we obtain that 
$\mu=\widehat{\Theta}(u)$ and $g_j=  \widehat{A}_j(u)$ for every $1\leq j\leq N$. 
From \eqref{mkslu0}  we conclude that 
$$ \langle g,v\rangle=\sum_{j=1}^N \int_\Omega \widehat{A}_j(u)\,\partial_j v\,dx+\int_\Omega \widehat{\Theta}(u) \,v\,dx =\langle \mathcal Au,v\rangle +\langle \mathcal P_\Theta(u),v\rangle
$$ for every $v\in W_0^{1,\overrightarrow{p}}(\Omega)$. This proves that 
$g=(\mathcal A +\mathcal P_\Theta)\,u$, namely, \eqref{pm30} holds.   

In conclusion, by satisfying the M type condition in Definition~\ref{nsr}, the operator $\mathcal A +\mathcal P_\Theta$ turns out to be pseudo-monotone. 
\end{proof}

\begin{cor} \label{pso} The operator 
	$\mathcal A+\mathcal P_\Theta-\mathfrak{B}:W_0^{1,\overrightarrow{p}}(\Omega)\to W^{-1,\overrightarrow{p}'}(\Omega)$ is pseudo-monotone.
	\end{cor}

\begin{proof} 
	The claim follows from Lemmas~\ref{strong0} and \ref{lem-tt11}, jointly with Proposition~\ref{prop_pmm}. 
	\end{proof}

\section{Proof of the first assertion in Theorem~\ref{nth2}} \label{sec6}

Here, we assume \eqref{IntroEq0}, \eqref{newlab0}, \eqref{ellip} and \eqref{cond1}, whereas $\mathfrak B$ belongs to the class $\mathfrak{BC}$.  
For every $\varepsilon>0$, we define $\Phi_\varepsilon(x,t,\xi):\Omega\times \mathbb R\times \mathbb R^N\to \mathbb R$ as follows
\begin{equation} \label{fina} \Phi_\varepsilon(x,t,\xi):=\frac{\Phi(x,t,\xi)}{1+\varepsilon\, |\Phi(x,t,\xi)|}
\end{equation}
for a.e. $x\in \Omega$ and all $ (t,\xi)\in \mathbb R\times \mathbb R^N$. 
For $\varepsilon>0$ fixed, $\Phi_\varepsilon$ satisfies the same properties as $\Phi$, that is, the sign-condition
and the growth condition in \eqref{cond1}. Moreover, $\Phi_\varepsilon$ becomes a bounded function, namely, for a.e. $x\in \Omega$ and every $(t,\xi)\in \mathbb R\times \mathbb R^N$,  
\begin{equation} \label{sig}
\Phi_\varepsilon(x,t,\xi)\,t\geq 0,\quad 
|\Phi_\varepsilon(x,t,\xi)|\leq \min\,\{ |\Phi(x,t,\xi)|,1/\varepsilon\}. 
\end{equation}
We consider approximate problems to  
\eqref{eq101} with 
$f=0$ and $\Phi$ replaced by 
$\Phi_\varepsilon$, that is,
\begin{equation} \label{dit2}
\left\{  \begin{aligned} 
&\mathcal A u_\varepsilon + \Phi_\varepsilon(x,u_\varepsilon,\nabla u_\varepsilon) +\Theta(x,u_\varepsilon,\nabla u_\varepsilon)=\mathfrak{B}u_\varepsilon\quad \mbox{in }  \Omega,\\ 
& u_\varepsilon\in W_0^{1,\overrightarrow{p}}(\Omega).
\end{aligned} \right. 
\end{equation}

\noindent As in Theorem~\ref{het}, by a solution of \eqref{dit2}, we mean a function
$u_\varepsilon\in  W_0^{1,\overrightarrow{p}}(\Omega)$ such that 
\begin{equation} \label{ssdili}
\sum_{j=1}^N \int_\Omega \widehat{A}_j(u_\varepsilon)\,\partial_j v\,dx+ \int_\Omega \widehat{\Phi}_\varepsilon(u_\varepsilon) \,v\,dx+
\int_\Omega \widehat{\Theta}(u_\varepsilon)\,v\,dx=\langle \mathfrak{B} u_\varepsilon,v\rangle 
\end{equation} 
for every $v\in W_0^{1,\overrightarrow{p}}(\Omega)$, where for convenience we 
define
$$ \widehat{\Phi}_\varepsilon(u_\varepsilon)(x):=\Phi_\varepsilon(x,u_\varepsilon(x),\nabla u_\varepsilon(x))  
\quad \mbox{for a.e. }x\in \Omega. $$

\begin{lemma} \label{l1}
For every $\varepsilon>0$, there exists a solution 
$u_\varepsilon$ for \eqref{dit2}. Moreover, we have: 	

$(a) $ For a positive constant $C$, independent of $\varepsilon$, it holds 
\begin{equation}\label{c11}
\|u_\varepsilon\|_{W_0^{1,\overrightarrow{p}}(\Omega)}+ \int_\Omega \widehat\Phi_\varepsilon(u_\varepsilon)\,u_\varepsilon\,dx \leq C.
\end{equation}

$(b)$ There exists $U\in W_0^{1,\overrightarrow{p}}(\Omega)$ such that, up to a subsequence of $\{u_\varepsilon\}$,  
\begin{equation}\label{wco} 
u_\varepsilon\rightharpoonup U\ \mbox{(weakly) in } W_0^{1,\overrightarrow{p}}(\Omega)\quad \mbox{and}\quad 
u_\varepsilon \to U\ \mbox{a.e. in } \Omega\ \mbox{as }\varepsilon\to 0.
\end{equation}
\end{lemma}

\begin{proof} Let $\varepsilon>0$ be arbitrary. From \eqref{sig}, we see that
$\Phi_\varepsilon+\Theta$ satisfies the same assumptions  as $\Theta$ in Section~\ref{newso}. So, Theorem~\ref{het} applies with $\mathcal P_\Theta$ replaced by $\mathcal P_{\Theta,\varepsilon}$, where 
$$\langle  \mathcal P_{\Theta,\varepsilon}(u),v\rangle:=
\int_\Omega \left( \widehat \Theta(u)+\widehat{\Phi}_\varepsilon(u)\right)v\,dx 
\quad \mbox{for every } u,v\in W_0^{1,\overrightarrow{p}}(\Omega).   
$$ 
This means that \eqref{dit2} admits at least a solution $u_\varepsilon\in W_0^{1,\overrightarrow{p}}(\Omega)$ for every $\varepsilon>0$. 

\vspace{0.2cm}
$(a)$ 
By taking $v=u_\varepsilon$ in \eqref{ssdili}, we derive that
\begin{equation} \label{ion2} 
\langle \mathcal Au_\varepsilon+\mathcal P_\Theta
(u_\varepsilon)-\mathfrak{B} u_\varepsilon,u_\varepsilon\rangle +\int_\Omega \widehat{\Phi}_\varepsilon(u_\varepsilon) \,u_\varepsilon\,dx=0.
\end{equation}
Moreover, since $\mathfrak{B}$ is a bounded operator from $W_0^{1,\overrightarrow{p}}(\Omega)$ into its dual, it follows that for some constant $C_0>0$, we have 
$$\|\mathfrak Bu_\varepsilon\|_{W^{-1,\overrightarrow{p}'}(\Omega)}\leq C_0
\quad \mbox{for every } \varepsilon>0.
$$ 
Using \eqref{newlab0}, the coercivity condition in \eqref{ellip} and Young's inequality, we infer that for every $\delta>0$, there exists a constant $C_\delta>0$ such that 
\begin{equation} \label{strani}
\begin{aligned} 
\langle \mathcal Au_\varepsilon+\mathcal P_\Theta
(u_\varepsilon)-\mathfrak{B} u_\varepsilon,u_\varepsilon\rangle  
&\geq \nu_0 \sum_{j=1}^N \|\partial_j u_\varepsilon\|_{L^{p_j}(\Omega)}^{p_j}-(C_0+C_\Theta)\,\|u_\varepsilon\|_{W_0^{1,\overrightarrow{p}}(\Omega)}\\
&\geq (\nu_0-\delta) \sum_{j=1}^N \|\partial_j u_\varepsilon\|_{L^{p_j}(\Omega)}^{p_j}-C_\delta
\end{aligned}
\end{equation}
for every $\varepsilon>0$.  
Thus, using \eqref{ion2} and \eqref{strani}, jointly with \eqref{sig}, we arrive at 
$$ (\nu_0-\delta) \sum_{j=1}^N \|\partial_j u_\varepsilon\|_{L^{p_j}(\Omega)}^{p_j}\leq
(\nu_0-\delta) \sum_{j=1}^N \|\partial_j u_\varepsilon\|_{L^{p_j}(\Omega)}^{p_j}+\int_\Omega \widehat{\Phi}_\varepsilon(u_\varepsilon) \,u_\varepsilon\,dx\leq C_\delta. 
$$
By choosing $\delta\in (0,\nu_0)$, we readily conclude the assertion of 
\eqref{c11}.

\medskip
$(b)$ From \eqref{c11} and the reflexivity of $W_0^{1,\overrightarrow{p}}(\Omega)$,
we infer that, up to a subsequence, $u_\varepsilon$  converges weakly to some $U$ in $ W_0^{1,\overrightarrow{p}}(\Omega)$. Then, we conclude \eqref{wco} by using  
Remark~\ref{an-sob} in the Appendix, which implies that, up to a subsequence, 
$ u_\varepsilon\to U$ (strongly) in $ L^q(\Omega)$ if  
$q\in [1,p^\ast)$ and $ u_\varepsilon\to U$ a.e. in $\Omega$ as $\varepsilon\to 0$. 
\end{proof}

For the remainder of this section, $u_\varepsilon$ and $U$ have the meaning in Lemma~\ref{l1}.

\subsection{Strong convergence of $T_k(u_\varepsilon)$} \label{sect32}
For $v,w\in W_0^{1,\overrightarrow{p}}(\Omega)$ and a.e. $x\in \Omega$, we define
$\mathcal D_{u_\varepsilon}(v,w)(x)$ as in Subsection \ref{prel} in the Appendix, namely, 
\begin{equation} \label{epslom}
\mathcal D_{u_\varepsilon}(v,w)(x):=\sum_{j=1}^N \left[
A_j(x,u_\varepsilon(x),\nabla v(x))-A_j(x,u_\varepsilon(x),\nabla w(x))
\right] \partial_j (v-w)(x). 
\end{equation}
For any fixed integer $k\geq 1$, we obtain $\mathcal D_{u_\varepsilon}(T_k(u_\varepsilon),T_k(U))$ by replacing $v$ and $w$ in \eqref{epslom} by $T_k(u_\varepsilon)$ and $T_k(U)$, respectively. For simplicity, we write $\mathcal D_{\varepsilon,k}(x)$ instead of $\mathcal D_{u_\varepsilon}(T_k(u_\varepsilon),T_k(U))(x)$, that is,
\begin{equation} \label{dep}
\mathcal D_{\varepsilon,k}(x):=\sum_{j=1}^N \left[ A_j(x,u_\varepsilon,\nabla T_k (u_\varepsilon))-
A_j(x,u_\varepsilon,\nabla T_k (U)) \right] \partial_j  (T_k (u_\varepsilon)-T_k (U)).
\end{equation} 

\begin{lemma} \label{arc} 
There exists a subsequence of $\{u_\varepsilon\}$, relabeled $\{u_\varepsilon\}$, such that 
\begin{equation} \label{extra} 
\nabla u_\varepsilon\to \nabla U\ \mbox{a.e. in } \Omega\ \mbox{and } 
T_k (u_\varepsilon)\to T_k (U)\ \mbox{(strongly) in } W_0^{1,\overrightarrow{p}}(\Omega)\ \mbox{as } \varepsilon\to 0
\end{equation} for every integer $k\geq 1$. 
\end{lemma}

\begin{proof}  
Recall that $\{u_\varepsilon\}$ satisfies \eqref{wco} in Lemma~\ref{l1}. 
By a standard diagonal argument, it suffices to show that for every integer $k\geq 1$, there exists a subsequence $\{u_\varepsilon\}$ (depending on $k$ and relabeled $\{u_\varepsilon\}$) satisfying
\begin{equation}
\label{hsn3}
\nabla T_k (u_\varepsilon)\to \nabla T_k (U)\ \mbox{a.e. in } \Omega \quad \mbox{and}\quad 
T_k (u_\varepsilon)\to T_k (U)\ \mbox{(strongly) in } W_0^{1,\overrightarrow{p}}(\Omega).
\end{equation}
Moreover, in light of Lemma~\ref{joc1} in the Appendix, 
we conclude \eqref{hsn3} by showing that, for every integer $k\geq 1$, there exists   a subsequence of $\{u_\varepsilon\}$ (depending on $k$ and relabeled $\{u_\varepsilon\}$) such that 
\begin{equation} \label{zero} \mathcal D_{\varepsilon,k}
\to 0\ \mbox{ in } L^1(\Omega) \ \mbox{as } \varepsilon\to 0.\end{equation}  	

Let $k\geq 1$ be fixed. 	
Clearly, the monotonicity assumption in \eqref{ellip} yields that $\mathcal D_{\varepsilon,k}\geq 0$ a.e. in $\Omega$. Hence, to prove \eqref{zero}, it suffices to show that (up to a subsequence of 
$\{u_\varepsilon\}$), 
\begin{equation} \label{zeroo} \limsup_{\varepsilon\to 0} \int_\Omega 
\mathcal D_{\varepsilon,k}(x)\,dx\leq 0  .\end{equation} 
We define $z_{\varepsilon,k}$ as follows
$$z_{\varepsilon,k}:=T_k (u_\varepsilon)-T_k (U).$$
We observe that 
$$ \partial_j z_{\varepsilon,k}\, \chi_{\{|u_\varepsilon|\geq k\}}=-\partial_j T_k (U) \,   \chi_{\{|u_\varepsilon|\geq k\}}=
-\partial_j U\, \chi_{\{|u_\varepsilon|\geq k\}} \,\chi_{\{|U|<k\}}.
$$
Moreover, we see that 
\begin{equation} \label{fii}  \chi_{\{|u_\varepsilon|\geq k\}} \,\chi_{\{|U|<k\}}\to 0\ \ \mbox{a.e. in } \Omega\ \text{as } \varepsilon\to 0. 
\end{equation}
By the Dominated Convergence Theorem, for every $1\leq j\leq N$, we have 
\begin{equation} \label{dom} 
\partial_j U\, \chi_{\{|u_\varepsilon|\geq k\}} \,\chi_{\{|U|<k\}}\to 0\quad \mbox{(strongly) in } L^{p_j}(\Omega)\ \ \mbox{as } \varepsilon\to 0.\end{equation}
On the other hand, from the growth condition on $A_j$ in \eqref{ellip} and the {\em a priori} estimates in Lemma~\ref{l1}, we infer that $\{A_j(x,u_\varepsilon,\nabla T_k (u_\varepsilon))\}_{\varepsilon}$ and 
$\{A_j(x,u_\varepsilon,\nabla T_k (U))\}_\varepsilon$ are bounded in $L^{p_j'}(\Omega)$ and, hence, up to a subsequence of $\{u_\varepsilon\}$, they
converge weakly in $L^{p_j'}(\Omega)$ for each $1\leq j\leq N$. This, jointly with \eqref{dom}, gives that 
$$  \Xi_{j,\varepsilon,k}(x):= \left[ A_j(x,u_\varepsilon,\nabla T_k (u_\varepsilon))-
A_j(x,u_\varepsilon,\nabla T_k (U)) \right] \partial_j U  \chi_{\{|u_\varepsilon|\geq k\}} \,\chi_{\{|U|<k\}}
$$ converges to $0$ in $L^1(\Omega)$ 
as $\varepsilon\to 0$ for every $1\leq j\leq N$. It follows that 
$$ \int_\Omega \mathcal D_{\varepsilon,k}(x) \,\chi_{\{|u_\varepsilon|\geq k\}}\,dx
=-\sum_{j=1}^N\int_\Omega 
\Xi_{j,\varepsilon,k}(x)\,dx\to 0\quad \text{as }\varepsilon\to 0.
$$ Thus, to conclude \eqref{zeroo}, it remains to show that 
\begin{equation} \label{zer} \limsup_{\varepsilon\to 0} \int_\Omega \mathcal D_{\varepsilon,k}(x)\, \chi_{\{|u_\varepsilon|< k\}}\,dx\leq 0  .\end{equation} 
{\em Proof of \eqref{zer}.} 
We define $\varphi_\lambda:\mathbb R\to \mathbb R$ as follows
$$  \varphi_\lambda(t)=t \exp\,(\lambda t^2) \quad \mbox{for every } t\in \mathbb R.
$$ 
We choose $\lambda=\lambda(k)>0$ large such that $ 4\nu_0^2 \,\lambda>\phi^2(k)$, where $\phi$ appears in the growth assumption on $\Phi$, see \eqref{cond1}. This choice of $\lambda$ ensures that for every $ t\in \mathbb R$ 
\begin{equation} \label{bea}  \lambda t^2-\frac{\phi(k)}{2\nu_0} |t|+\frac{1}{4}>0 \quad  \text{and, hence,}\quad 
 \varphi_\lambda'(t) -\frac{\phi(k)}{\nu_0} \,| \varphi_\lambda(t)| >\frac{1}{2}. 
\end{equation} 
For $v\in W_0^{1,\overrightarrow{p}}(\Omega) $, we define 
\begin{equation*} \label{epi} \mathcal E_{\varepsilon,k} (v)=
\sum_{j=1}^N \int_\Omega A_j(x,u_\varepsilon,\nabla v)\partial_j z_{\varepsilon,k} \left[ \varphi_\lambda'(z_{\varepsilon,k}) 
-\frac{\phi(k)}{\nu_0} \,| \varphi_\lambda(z_{\varepsilon,k})|\right] \chi_{\{|u_\varepsilon|< k\}}\,dx.
\end{equation*}
Returning to the definition of $ \mathcal D_{\varepsilon,k}$ in \eqref{dep} and using \eqref{bea}, we arrive at 
\begin{equation} \label{lig} \frac{1}{2}   \int_\Omega \mathcal D_{\varepsilon,k}(x) \,\chi_{\{|u_\varepsilon|< k\}}\,dx
\leq  \mathcal E_{\varepsilon,k} (T_k (u_\varepsilon))- \mathcal E_{\varepsilon,k} (T_k (U)).
\end{equation}
Since $T_k (u_\varepsilon)=u_\varepsilon$ on the set $\{|u_\varepsilon|<k\}$, in light of \eqref{lig}, we complete the proof of \eqref{zer} by showing that 
\begin{eqnarray} 
& \displaystyle\lim_{\varepsilon\to 0}  \mathcal E_{\varepsilon,k} (T_k (U))=0, \label{lom}\\
&\displaystyle \limsup_{\varepsilon\to 0}  \mathcal E_{\varepsilon,k} (u_\varepsilon)\leq 0.
\label{tir}
\end{eqnarray}

\noindent {\em Proof of \eqref{lom}.} For each $1\leq j\leq N$, the growth condition in \eqref{ellip} gives a nonnegative function $F_j\in L^{p_j'}(\Omega)$ such that on the set $\{|u_\varepsilon|<k\}$, we have 
$|A_j(x,u_\varepsilon,\nabla T_k(U))|\leq F_j$ for every $\varepsilon>0$. Since $|z_{\varepsilon,k}|\leq 2k$, we can find a constant $C_k>0$ such that 
$$ \left| \varphi_\lambda'(z_{\varepsilon,k}) 
-\frac{\phi(k)}{\nu_0} \,| \varphi_\lambda(z_{\varepsilon,k})| \right| \leq C_k.
$$
On the other hand, for each $1\leq j\leq N$, we have 
$$  \partial_j z_{\varepsilon,k} \,\chi_{\{|u_\varepsilon|< k\}}= \partial_j z_{\varepsilon,k}+\partial_j U\,\chi_{\{|U|< k\}}
\chi_{\{|u_\varepsilon|\geq  k\}}.
$$ This, together with \eqref{dom} and the weak convergence of $\partial_j z_{\varepsilon,k}$ to $0$ in $L^{p_j}(\Omega)$ as $\varepsilon\to 0$,
implies that $ \partial_j z_{\varepsilon,k} \,\chi_{\{|u_\varepsilon|< k\}} $ converges weakly to $0$ in $L^{p_j}(\Omega)$ as $\varepsilon\to 0$. Hence, we have  
$$  |\mathcal E_{\varepsilon,k} (T_k (U))|\leq C_k \sum_{j=1}^N\int_{\Omega} F_j \, |\partial_j z_{\varepsilon,k} |\,\chi_{\{|u_\varepsilon|< k\}} \,dx\to 0\quad \mbox{as } \varepsilon\to 0,
$$ which proves \eqref{lom}.

\vspace{0.2cm}
\noindent {\em Proof of \eqref{tir}.} 
From \eqref{wco}, we have 
$$z_{\varepsilon,k}\to 0\ \mbox{a.e. in } \Omega\ \mbox{and}\ 
z_{\varepsilon,k} 
\rightharpoonup  0\ \mbox{(weakly) in } W_0^{1,\overrightarrow{p}}(\Omega) \ \mbox{ as }\varepsilon\to 0.$$ Since $|z_{\varepsilon,k}|\leq 2k$ a.e. in $\Omega$,
we get $\varphi_\lambda(z_{\varepsilon,k})\in W_0^{1,\overrightarrow{p}}(\Omega)\cap L^\infty(\Omega) $. Moreover,  
\begin{equation} \label{ver} \varphi_\lambda(z_{\varepsilon,k})\to 0\ \text{ a.e. in }
\Omega\ \text{and } \varphi_\lambda(z_{\varepsilon,k}) \rightharpoonup  0\ \text{ (weakly) in } W_0^{1,\overrightarrow{p}}(\Omega) \ 
\text{as }\varepsilon\to 0.\end{equation}
Observe that $u_\varepsilon\, z_{\varepsilon,k}\geq 0$ on the set $\{|u_\varepsilon|\geq k\}$, which gives that 
$$ \widehat \Phi_\varepsilon(u_\varepsilon) \,\varphi_\lambda(z_{\varepsilon,k})\,\chi_{\{|u_\varepsilon|\geq k\}}\geq 0. 
$$ 
Thus, by testing \eqref{ssdili} with $v=\varphi_\lambda(z_{\varepsilon,k})$,  
we obtain that 
\begin{equation} \label{nou} 
\langle \mathcal A u_\varepsilon, \varphi_\lambda(z_{\varepsilon,k})\rangle +
\int_\Omega \widehat\Phi_\varepsilon(u_\varepsilon) \,\varphi_\lambda(z_{\varepsilon,k})\,\chi_{\{|u_\varepsilon|<k\}}\,dx\leq 
\langle \mathfrak{B} u_\varepsilon, \varphi_\lambda(z_{\varepsilon,k})\rangle
-\int_\Omega \widehat{\Theta}(u_\varepsilon)\,\varphi_\lambda(z_{\varepsilon,k})\,dx.
\end{equation} 
To simplify exposition, we now introduce some notation:
\begin{equation*} \label{nota}
\begin{aligned}
& X_k(\varepsilon):= \phi(k)  \, \int_\Omega \left[ \frac{1}{\nu_0}\, \sum_{j=1}^N \widehat A_j(u_\varepsilon)\,\partial_j (T_k U) +c(x)\right] |\varphi_\lambda(z_{\varepsilon,k})| \,\chi_{\{|u_\varepsilon|<k\}}\,dx,\\
& Y_k(\varepsilon):= \sum_{j=1}^N \int_\Omega \widehat A_j(u_\varepsilon)\,\partial_j U\, \varphi_\lambda'(z_{\varepsilon,k})\,\chi_{\{|U|<k\}}\,
\chi_{\{|u_\varepsilon|\geq k\} }\,dx.
\end{aligned}
\end{equation*}
We rewrite the first term in the left-hand side of \eqref{nou} as follows
\begin{equation} \label{cop}   \langle \mathcal A u_\varepsilon, \varphi_\lambda(z_{\varepsilon,k})\rangle=
\sum_{j=1}^N \int_\Omega \widehat A_j( u_\varepsilon) \partial_j z_{\varepsilon,k} \,\varphi_\lambda'(z_{\varepsilon,k})\,\chi_{\{|u_\varepsilon|<k\}}\,dx-Y_k(\varepsilon).
\end{equation}
The coercivity condition in \eqref{ellip} and the growth condition of $\Phi$ in \eqref{cond1} imply that 
\begin{equation} \label{loc}  |\widehat \Phi_\varepsilon(u_\varepsilon) |
\,\chi_{\{|u_\varepsilon|<k\}}\leq 
\phi(k) \left[ \frac{1}{\nu_0} \sum_{j=1}^N \widehat A_j(u_\varepsilon) \partial_j u_\varepsilon +c(x)\right] 
\,\chi_{\{|u_\varepsilon|<k\}}. \end{equation} 
In the right-hand side of \eqref{loc} we replace $\partial_j u_\varepsilon$ by $\partial_j z_{\varepsilon,k}+\partial_j T_k (U)$, then we multiply the inequality by
$| \varphi_\lambda(z_{\varepsilon,k})|$ and integrate over $\Omega$ with respect to $x$. It follows that the second term in the left-hand side of \eqref{nou} is at least 
$$  -\frac{\phi(k)}{\nu_0} \sum_{j=1}^N\int_\Omega \widehat A_j(u_\varepsilon) \partial_j z_{\varepsilon,k} \,
| \varphi_\lambda(z_{\varepsilon,k})| \,\chi_{\{|u_\varepsilon|<k\}}\,dx-X_{k}(\varepsilon).
$$ 
Using this fact, as well as \eqref{cop}, in \eqref{nou}, we see that $\mathcal E_{\varepsilon,k}(u_\varepsilon)$  satisfies the estimate
\begin{equation} \label{veb}  \mathcal E_{\varepsilon,k}(u_\varepsilon)\leq X_k(\varepsilon)+Y_k(\varepsilon)+\langle \mathfrak{B} u_\varepsilon, \varphi_\lambda(z_{\varepsilon,k})\rangle-\int_\Omega \widehat{\Theta}(u_\varepsilon)\,\varphi_\lambda(z_{\varepsilon,k})\,dx.
\end{equation} 

\medskip
To conclude the proof of \eqref{tir}, it suffices to show that each term in the right-hand side of \eqref{veb} converges to $0$ as $\varepsilon\to 0$. 
Recall that $\varphi_\lambda(z_{\varepsilon,k})\in W_0^{1,\overrightarrow{p}}(\Omega)\cap L^\infty(\Omega) $ satisfies \eqref{ver}. 
Thus, using \eqref{newlab0} and the property $(P_2)$ of $\mathfrak{B}$, we get that the third, as well as the fourth, term in the right-hand side of \eqref{veb} converges to zero as $\varepsilon\to 0$. 

\noindent We next look at $X_k(\varepsilon)$. In view of the pointwise convergence in \eqref{ver} and $c\in L^1(\Omega)$, we infer from the Dominated Convergence Theorem that 
\begin{equation} \label{alb1} c(x) |\varphi_\lambda(z_{\varepsilon,k})|\,\chi_{\{|u_\varepsilon|<k\}}\to 0 \ \mbox{in } L^1(\Omega)\ \mbox{as } \varepsilon\to 0.\end{equation} 
Next, up to a subsequence of $\{u_\varepsilon\}$, we find that $\widehat A_j(u_\varepsilon) $ converges weakly 
in $L^{p_j'}(\Omega)$ 
as $\varepsilon \to 0$ for every $1\leq j\leq N$ using
the boundedness of $\widehat{A}_j:W_0^{1,\overrightarrow{p}}(\Omega)\to L^{p_j'}(\Omega)$ (see Lemma~\ref{lem-tt00}). Hence,  
$ \sum_{j=1}^N \widehat A_j(u_\varepsilon)\,\partial_j  U $
converges in $L^1(\Omega)$ as $\varepsilon\to 0$. Then, there exists a nonnegative function $F\in L^1(\Omega)$ (independent of $\varepsilon$) such that,  
up to a subsequence of $\{u_\varepsilon\}$, we have 
\begin{equation} \label{fiu}  \left|\sum_{j=1}^N \widehat A_j(u_\varepsilon)\,\partial_j  U\right|\leq F \quad \mbox{a.e. in } \Omega\ \mbox{for every } \varepsilon>0.\end{equation}
We can now again use the Dominated Convergence Theorem to conclude that 
\begin{equation} \label{alb2}  \sum_{j=1}^N \widehat A_j(u_\varepsilon)\, \partial_j T_k (U)\, |\varphi_\lambda(z_{\varepsilon,k})|\, \chi_{\{|u_\varepsilon|<k\}}\to 0
\ \mbox{in } L^1(\Omega)\ \mbox{as } \varepsilon\to 0.  \end{equation} 
From \eqref{alb1} and \eqref{alb2}, we find that $\lim_{\varepsilon\to 0} X_k(\varepsilon)=0$. 
Since $|\varphi_\lambda'(z_{\varepsilon,k})|$ is bounded above by a constant independent of $\varepsilon$ (but dependent on $k$), we can use a similar argument, based on \eqref{fii} and \eqref{fiu}, to obtain that, up to a subsequence of $\{u_\varepsilon\}$, 
$\lim_{\varepsilon\to 0} Y_k(\varepsilon)=0$. This ends the proof of the convergence to zero of the right-hand side of \eqref{veb} as $\varepsilon\to 0$. Consequently, the proof of \eqref{tir}, and thus of \eqref{zer}, is complete.  
\end{proof}

\subsection{Passing to the limit} \label{pass}

From now on, the meaning of $\{u_\varepsilon\}_\varepsilon$ is given by Lemma~\ref{arc}. 
Using Lemma~\ref{l1}, we 
prove in Lemma~\ref{cor2} that $U$ is a solution of \eqref{eq101} with $f=0$ and, moreover, $U$ satisfies all the properties stated in Theorem~\ref{nth2} (i). Besides \eqref{extra}, the other fundamental property that allows us to pass to the limit as
$\varepsilon\to 0$ in \eqref{ssdili} for every $v\in W_0^{1,\overrightarrow{p}}(\Omega)\cap L^\infty(\Omega)$ is the following convergence 
\begin{equation} \label{cov}
\widehat\Phi_{\varepsilon} (u_\varepsilon)\to \widehat\Phi(U)\ \mbox{ (strongly) in } L^1(\Omega)\ \mbox{as } \varepsilon
\to 0.
\end{equation} 
The proof of \eqref{cov} is the main objective of our next result.

\begin{lemma} \label{coro1} We have $\widehat{\Phi}(U) \,U^{j}\in L^1(\Omega)$ for $j=0,1$ and \eqref{cov} holds. 
\end{lemma}	

\begin{proof}
From the pointwise convergence 
$u_\varepsilon  \to U$ and $  
\nabla u_\varepsilon\to \nabla U$ a.e. in $\Omega$ as $\varepsilon\to 0,$ jointly with the fact that $\Phi(x,t,\xi):\Omega\times \mathbb R\times \mathbb R^N\to \mathbb R$ is a Carath\'eodory function, we infer that 
$\widehat{\Phi}(u_\varepsilon)  
\to \widehat{\Phi}(U)$ and  
$  \widehat{\Phi}_\varepsilon(u_\varepsilon)\,u_\varepsilon\to \widehat{\Phi}(U)\,U$ a.e. in $\Omega$ as $ \varepsilon\to 0$.
Using this fact and that 
$\{\widehat{\Phi}_\varepsilon(u_\varepsilon)\,u_\varepsilon\}_\varepsilon$ is a sequence of nonnegative functions that is uniformly bounded in $L^1(\Omega)$ with respect to $\varepsilon$ (from Lemma~\ref{l1}), by Fatou's Lemma we conclude 
that $$\widehat{\Phi}(U) \,U\in L^1(\Omega).$$ 
This and the growth condition in \eqref{cond1} yield that $\widehat{\Phi}(U)\in L^1(\Omega)$. Indeed, for any $M>0$, on the set $ \Omega\cap \{|U|\leq M\}$, we have 
$|\widehat{\Phi}(U) |\leq \phi(M) \left(\sum_{j=1}^N |\partial_j U|^{p_j}+c(x)\right)\in L^1(\Omega).$ 
In turn, on the set $\Omega\cap \{|U|>M\}$, 
it holds
$ |\widehat{\Phi}(U) |\leq M^{-1}\, \widehat{\Phi}(U)\,U \in L^1(\Omega). $

\medskip
To finish the proof of Lemma~\ref{coro1}, it remains to establish \eqref{cov}. 

\medskip
\noindent {\em Proof of \eqref{cov}.} 
Since 
$\widehat{\Phi}_\varepsilon(u_\varepsilon)\to \widehat{\Phi}(U)$ a.e. in $\Omega$ as $\varepsilon\to 0$ and $\widehat{\Phi}(U)\in L^1(\Omega)$, by Vitali's Theorem, it suffices to show that $\{\widehat{\Phi}_\varepsilon(u_\varepsilon)\}_\varepsilon$ is uniformly integrable over $\Omega$. We next check this fact. For every $M>0$, we define
$$  D_{\varepsilon,M}:=\{ |u_\varepsilon|\leq M\}\ \quad\text{and}\quad E_{\varepsilon,M}:=\{ |u_\varepsilon|> M\}.
$$
For every $x\in D_{\varepsilon,M}$, using the growth condition of $\Phi$ in \eqref{cond1}, we find that 
$$|\widehat{\Phi}_\varepsilon(u_\varepsilon)(x)|\leq 
|\widehat{\Phi}(u_\varepsilon)(x)|\leq
\phi(M) \left( \sum_{j=1}^N |\partial_j T_M (u_\varepsilon)|^{p_j}+c(x)\right),$$ 
with $c \in L^1(\Omega)$. Let $\omega$ be 
any measurable subset  of $\Omega$.  
It follows that  
$$ \int_{\omega\cap D_{\varepsilon,M}} |\widehat{\Phi}_\varepsilon(u_\varepsilon)|\,dx\leq 
\phi(M)\left(\sum_{j=1}^N  \|\partial_j (T_M u_\varepsilon)\|^{p_j}_{L^{p_j}(\omega)}+
\int_{\omega} c(x)\,dx
\right).
$$
On the other hand, using \eqref{c11} in Lemma~\ref{l1}, we see that 
$$  \int_{\omega\cap E_{\varepsilon,M}}  
|\widehat{\Phi}_\varepsilon(u_\varepsilon)|
\,dx\leq  \frac{1}{M} \int_{\omega \cap E_{\varepsilon,M}} \widehat{\Phi}_\varepsilon(u_\varepsilon) \,u_\varepsilon\,dx
\leq \frac{C}{M}, 
$$ where $C>0$ is a constant independent of $\varepsilon$ and $\omega$. 
Consequently, we find that 
\begin{equation}\label{new1}
\int_\omega |\widehat{\Phi}_\varepsilon(u_\varepsilon)|\,dx \leq
\phi(M)\left(\sum_{j=1}^N  \|\partial_j (T_M u_\varepsilon)\|^{p_j}_{L^{p_j}(\omega)}+
\int_{\omega} c(x)\,dx
\right)
+\frac{C}{M}. 
\end{equation}
Lemma~\ref{arc} yields that $\partial_j T_M (u_\varepsilon)\to \partial_j T_M (U)$ (strongly) in $L^{p_j}(\Omega)$ as $\varepsilon\to 0$ for every $1 \le j \le N$. Since $c\in L^1(\Omega)$, from \eqref{new1} we get the uniform integrability of 
$\{\widehat{\Phi}_\varepsilon(u_\varepsilon)\}_\varepsilon$ over $\Omega$. We end the proof of \eqref{cov} by Vitali's Theorem. 
\end{proof}

By Lemma~\ref{coro1}, to finish the proof of Theorem~\ref{nth2} (i), we need to show the following.

\begin{lemma} \label{cor2} 
The function $U$ is a solution of \eqref{eq101} with $f=0$ and, moreover,  \eqref{sesim} holds for $v=u=U$.
\end{lemma}

\begin{proof} Fix $v\in W_0^{1,\overrightarrow{p}}(\Omega)\cap L^\infty(\Omega)$ arbitrary. Since $u_\varepsilon$ is a solution of \eqref{dit2}, we have 
\begin{equation} \label{sis}
\sum_{j=1}^N \int_\Omega \widehat{A}_j(u_\varepsilon)\,\partial_j v\,dx+
\int_\Omega \widehat{\Phi}_\varepsilon(u_\varepsilon)\,v\,dx+
\int_\Omega \widehat{\Theta}(u_\varepsilon)\,v\,dx=\langle \mathfrak{B} u_\varepsilon,v\rangle .
\end{equation}
By Lemma~\ref{coro1}, the second term in the left-hand side of \eqref{sis} converges to 
$\int_\Omega \widehat{\Phi}(U)\,v$ as $\varepsilon\to 0$, whereas the right-hand side of \eqref{sis} converges to $\langle \mathfrak{B} U,v\rangle$ 
based on the weak convergence of $u_\varepsilon$ to $U$ in $W_0^{1,\overrightarrow{p}}(\Omega)$ as $\varepsilon\to 0$. 
Using \eqref{wco} and \eqref{extra},
we find that  
\begin{equation} \label{milop} \widehat{\Theta}(u_\varepsilon)\to \widehat{\Theta}(U)\quad \mbox{and}\quad 
\widehat{A}_j(u_\varepsilon)\to 
\widehat{A}_j(U)\ \mbox{a.e. in } \Omega\ \mbox{for } 1\leq j\le N.
\end{equation}
Thus, in light of \eqref{newlab0}, and the Dominated Convergence Theorem, we obtain that 
$$ \int_\Omega \widehat{\Theta}(u_\varepsilon)\,v\,dx\to  
\int_\Omega \widehat{\Theta}(U)\,v\,dx\quad \mbox{as } \varepsilon\to 0. 
$$ 
Since $\{\widehat{A}_j(u_\varepsilon)\}_\varepsilon$ is uniformly bounded in $L^{p_j'}(\Omega)$ with respect to $\varepsilon$, we observe from \eqref{milop} that (up to a subsequence) 
$\widehat{A}_j(u_\varepsilon) \rightharpoonup 
\widehat{A}_j(U)$ (weakly) in  $L^{p_j'}(\Omega)$ as $\varepsilon\to 0$   
for each $1\leq j\leq N$. It follows that 
\begin{equation*} \label{von} \sum_{j=1}^N \int_\Omega \widehat{A}_j(u_\varepsilon)\,\partial_j v\,dx\to \sum_{j=1}^N 
\int_\Omega \widehat{A}_j(U)\,\partial_j v\,dx \ \ 
\mbox{as } \varepsilon\to 0.
\end{equation*}
By letting $\varepsilon\to 0$ in \eqref{sis}, we conclude that 
\begin{equation} \label{sesi}
\sum_{j=1}^N \int_\Omega \widehat{A}_j(U)\,\partial_j v\,dx+
\int_\Omega \widehat{\Phi}(U)\,v\,dx+ \int_\Omega \widehat{\Theta}(U)\,v\,dx=\langle \mathfrak{B} U,v\rangle 
\end{equation}
for every $v\in W_0^{1,\overrightarrow{p}}(\Omega)\cap L^\infty(\Omega)$. Hence, $U$ is a solution of \eqref{eq101} with $f=0$.  

\medskip
It remains to prove \eqref{sesim} for $v=u=U$.
Since $U$ may not be in $L^\infty(\Omega)$, we cannot directly use $v=U$ in \eqref{sesi}. Nevertheless, 
for every $k>0$, we have $T_k(U)\in W_0^{1,\overrightarrow{p}}(\Omega) \cap L^\infty(\Omega)$. Hence, by taking $v=T_k (U)$ in \eqref{sesi}, we have
\begin{equation} \label{lac}  
\langle \mathcal AU, T_k (U)\rangle + \int_\Omega  \widehat{\Phi}(U)\,T_k (U)\,dx
+ \int_\Omega \widehat{\Theta}(U)\,T_k (U)\,dx=
\langle \mathfrak{B}U,T_k(U)\rangle.
\end{equation}
Notice that $\|T_k (U)\|_{W_0^{1,\overrightarrow{p}}(\Omega)}\leq \|U\|_{W_0^{1,\overrightarrow{p}}(\Omega)}$ for all $k>0$. 
Moreover,  $\partial_j (T_k (U))\to \partial_j U$ a.e. in $\Omega$ as $k\to \infty$, for every $1\leq j\leq N$,  so that 
$  T_k (U)\rightharpoonup U$ (weakly) in $W_0^{1,\overrightarrow{p}}(\Omega)$ as $k\to \infty.$ 
Since $\mathcal A U$ and $\mathfrak{B} U$ belong to $W^{-1,\overrightarrow{p}'}(\Omega)$, it follows that 
\begin{equation*} \label{tas1} 
\lim_{k\to \infty} \langle \mathcal AU, T_k (U)\rangle = \langle \mathcal AU, U\rangle \quad \mbox{and } \quad \lim_{k \to \infty} \langle \mathfrak{B}U,T_k(U)\rangle=\langle \mathfrak{B} U, U\rangle. 
\end{equation*}
Recalling that $\widehat{\Phi}(U)\, U \in L^1(\Omega)$ and \eqref{newlab0} holds, from the Dominated Convergence Theorem, we can pass to the limit $k\to \infty$ in \eqref{lac} to conclude the proof. 
\end{proof}

\section{Proof of the second assertion in Theorem~\ref{nth2}} \label{lastsec}

Suppose for the moment only \eqref{IntroEq0}, \eqref{newlab0}, \eqref{ellip}, and \eqref{cond1}. Let $\mathfrak{B}$ be in the class $\mathfrak{BC}$.  
Overall, to prove Theorem~\ref{nth2} (ii), we follow similar arguments to those developed for proving Theorem~\ref{nth2} (i) in Section~\ref{sec6}. 
But there are several differences that appear when introducing a function $f\in L^1(\Omega)$
in the equation \eqref{eq101}. We first approximate $f$ by a ``nice" function $f_\varepsilon\in L^\infty(\Omega)$ with the properties that 
\begin{equation} \label{nice} |f_\varepsilon|\leq |f|\ \mbox{a.e. in }\Omega\ \mbox{ and } f_\varepsilon\to f\ \mbox{ a.e. in } \Omega\ \mbox{ as }\varepsilon\to 0.\end{equation} Then, by the Dominated Convergence Theorem, we find that 
\begin{equation} \label{conv1} f_\varepsilon\to f \ \mbox{(strongly) in } L^1(\Omega)\ \mbox{ as }\varepsilon\to 0.\end{equation}
For example, for every $\varepsilon>0$, we could take $f_\varepsilon(x):=f(x)/(1+\varepsilon |f(x)|)$ for a.e. $ x\in \Omega$. This approximation is done so that we can apply Theorem~\ref{nth2} (i) for the problem generated by \eqref{eq101} with $f_\varepsilon$ in place of $f$. 
Then such an approximate problem admits at least a solution $U_\varepsilon$, namely, 
\begin{equation} \label{eqn}
\left\{  \begin{aligned} 
&\mathcal A U_\varepsilon +\widehat\Phi(U_\varepsilon)+\widehat\Theta(U_\varepsilon) =
\mathfrak{B}U_\varepsilon+f_\varepsilon\quad \mbox{in } \Omega,\\ 
& U_\varepsilon\in W_0^{1,\overrightarrow{p}}(\Omega), \quad \widehat\Phi(U_\varepsilon)\in L^1(\Omega).
\end{aligned} \right.
\end{equation}
To see this, we observe that $\mathfrak{B}_\varepsilon: W_0^{1,\overrightarrow{p}}(\Omega)\to  W^{-1,\overrightarrow{p}'}(\Omega) $ belongs to the class $\mathfrak{BC}$, where \begin{equation} \label{bep} \langle  \mathfrak{B}_\varepsilon u,v\rangle=\langle  \mathfrak{B} u,v\rangle +\int_\Omega f_\varepsilon \,v\,dx\quad \mbox{for every }
u,v \in W_0^{1,\overrightarrow{p}}(\Omega).
\end{equation} 
By Theorem~\ref{nth2} (i) applied for $\mathfrak{B}_\varepsilon$ instead of $\mathfrak{B}$, we obtain a solution $U_\varepsilon$ for \eqref{eqn}.   
Thus,
\begin{equation} \label{sec}
\sum_{j=1}^N \int_\Omega \widehat{A}_j(U_\varepsilon)\, \partial_j  v\,dx
+\int_\Omega \widehat{\Phi}(U_\varepsilon)\,v\,dx
+\int_\Omega \widehat{\Theta}(U_\varepsilon)\,v\,dx
=\langle \mathfrak{B} U_\varepsilon,v\rangle +\int_\Omega f_\varepsilon\,v\,dx
\end{equation}
for every $v\in W_0^{1,\overrightarrow{p}}(\Omega)\cap L^\infty(\Omega)$. However, unlike Theorem~\ref{nth2} (i), to obtain that $U_\varepsilon$ is uniformly bounded in $W_0^{1,\overrightarrow{p}}(\Omega)$ with respect to $\varepsilon$, we need the following:

(i) $\mathfrak{B}$ to satisfy the extra condition $(P_3)$, that is, $\mathfrak B$ is chosen in the class $\mathfrak{BC}_+$; 

(ii) the additional hypothesis \eqref{info}, which we recall below:  

\noindent there exist positive constants $\tau$ and $\gamma$ such that 
for a.e. $x\in \Omega$ and every $\xi\in \mathbb R^N$
\begin{equation} \label{inf}
|\Phi(x,t,\xi)| \geq \gamma \sum_{j=1}^N |\xi_j|^{p_j} \quad \text{for all } |t|\geq \tau.
\end{equation}  

Without any loss of generality, we can assume $\tau>0$ large such that $\tau\gamma\geq \nu_0$, where $\nu_0$ appears in the coercivity condition of \eqref{ellip}. 

For the rest of this section, besides  \eqref{IntroEq0}, \eqref{newlab0}, \eqref{ellip} and \eqref{cond1}, we also assume (i) and (ii) above. 
To avoid repetition, we understand that all the computations in Section~\ref{sec6} are done here replacing $u_\varepsilon$, $U$ and $\Phi_\varepsilon$ by $U_\varepsilon$, $U_0$ and $\Phi$, respectively. We only stress the differences that appear compared with the developments in Section~\ref{sec6}.

\subsection{{\em A priori} estimates}
In Lemma~\ref{l1} we gave {\em a priori} estimates for the solution $u_\varepsilon$ to \eqref{dit2}, corresponding to the problem \eqref{eq101} with $f=0$ and $\Phi_\varepsilon$ instead of $\Phi$. We next get {\em a priori} estimates for $U_\varepsilon$ solving \eqref{eqn}, that is, \eqref{eq101} with $f_\varepsilon$ instead of $f$. 

\begin{lemma} \label{lem-ad}
Let $U_\varepsilon$ be a solution  of \eqref{eqn}. 

$(a) $ For a positive constant $C$, independent of $\varepsilon$, we have 
\begin{equation}\label{cas}
\|U_\varepsilon\|_{W_0^{1,\overrightarrow{p}}(\Omega)}+ \int_\Omega |\widehat{\Phi}(U_\varepsilon)|\,dx \leq C.
\end{equation}

$(b)$ There exists $U_0\in W_0^{1,\overrightarrow{p}}(\Omega)$ such that, up to a subsequence of $\{U_\varepsilon\}$,  
\begin{equation}\label{wwa} 
U_\varepsilon\rightharpoonup U_0\ \mbox{(weakly) in } W_0^{1,\overrightarrow{p}}(\Omega),\quad 
U_\varepsilon \to U_0\ \mbox{a.e. in } \Omega\ \mbox{as }\varepsilon\to 0.
\end{equation}
\end{lemma}

\begin{proof}  
$(a)$
The choice of $f_\varepsilon$ gives that $\|f_\varepsilon\|_{L^1(\Omega)}\leq \|f\|_{L^1(\Omega)}$. 
Let $\tau>0$ be as in \eqref{inf}. 
We have $\partial_j T_\tau(U_\varepsilon)=\chi_{\{|U_\varepsilon|<\tau\}} \,\partial_j U_{\varepsilon}$ a.e. in $ \Omega$ for every $1\leq j\leq N$. 
We now define 
$$  K_{\tau,\varepsilon}:=
\sum_{j=1}^N \int_\Omega \widehat{A}_j(U_\varepsilon)\,\partial_j U_\varepsilon \,\chi_{\{|U_\varepsilon|<\tau\}}\,dx+
\tau\int_\Omega |\widehat{\Phi}(U_\varepsilon)|\,\chi_{\{|U_\varepsilon|\geq \tau\}}\,dx-\langle \mathfrak{B} U_\varepsilon,T_\tau (U_\varepsilon)\rangle. 
$$
By taking $v=T_\tau(U_\varepsilon)\in  
W_0^{1,\overrightarrow{p}}(\Omega)\cap L^\infty(\Omega)$ in \eqref{sec} and using the sign-condition of $\Phi$ in \eqref{cond1}, we obtain that 
\begin{equation} \label{vio}  
K_{\tau,\varepsilon}\leq 
 \tau \left(\|f\|_{L^1(\Omega)}+C_\Theta \,{\rm meas}\,(\Omega)\right).
\end{equation}
By virtue of \eqref{inf} and the coercivity condition in \eqref{ellip}, we see that 
$$  \nu_0 \sum_{j=1}^N \int_\Omega |\partial_j U_\varepsilon|^{p_j} \,\chi_{\{|U_\varepsilon|<\tau\}}\,dx+
\tau\gamma  \sum_{j=1}^N \int_\Omega |\partial_j U_\varepsilon|^{p_j} \,\chi_{\{|U_\varepsilon|\geq \tau\}}\,dx
-\langle \mathfrak{B} U_\varepsilon,T_\tau (U_\varepsilon)\rangle\leq K_{\tau,\varepsilon}.  
$$
By our choice of $\tau$, we have $\tau\gamma>\nu_0$. Then, the above estimates lead to 
\begin{equation*} \label{gaz} \nu_0  \sum_{j=1}^N \int_\Omega |\partial_j U_\varepsilon|^{p_j}\,dx -\langle \mathfrak{B} U_\varepsilon,T_\tau (U_\varepsilon)\rangle\leq  
\tau \left(\|f\|_{L^1(\Omega)}+C_\Theta \,{\rm meas}\,(\Omega)\right).
\end{equation*}
This fact, jointly with the property $(P_3)$, gives the boundedness of $\{U_\varepsilon\}_{\varepsilon>0}$ in $W_0^{1,\overrightarrow{p}}(\Omega)$. Since $\mathfrak B$ is a bounded operator from $W_0^{1,\overrightarrow{p}}(\Omega)$ into its dual, we have 
$|\langle \mathfrak{B}U_\varepsilon,T_\tau (U_\varepsilon)|\leq C_1$, where $C_1$ is a positive constant independent of $\varepsilon$. Using \eqref{vio}, we find that 
\begin{equation} \label{nuc1}  \int_\Omega |\widehat{\Phi}(U_\varepsilon)|\,\chi_{\{|U_\varepsilon|\geq \tau\}}\,dx \leq C_1 \tau^{-1}+ \|f\|_{L^1(\Omega)}+C_\Theta \,{\rm meas}\,(\Omega):=C_2.
\end{equation}
Now, using the growth condition of $\Phi$ in \eqref{cond1}, we obtain a positive constant $C_3$ such that  
$ \int_\Omega |\widehat{\Phi}(U_\varepsilon)|\,\chi_{\{|U_\varepsilon|\leq \tau\}}\,dx \leq C_3$ 
for every $\varepsilon>0$. This completes the proof of \eqref{cas}. 

\medskip 
$(b)$ The assertion  in \eqref{wwa} follows from \eqref{cas} (see the proof of $(b)$ in Lemma~\ref{l1}). 
\end{proof}

\subsection{Strong convergence of $T_k(U_\varepsilon)$} The game plan is closely related to that in Subsection~\ref{sect32}. 
As mentioned before, when adapting the calculations, we need to replace $u_\varepsilon$, $U$ and $\mathfrak{B}$ in Section~\ref{sec6} by $U_\varepsilon$, $U_0$ and $\mathfrak{B}_\varepsilon$, respectively. 
The counterpart of Lemma~\ref{arc} holds so that we obtain the following. 

\begin{lemma} \label{cor42} There exists a subsequence of $\{U_\varepsilon\}_\varepsilon$, 
relabeled $\{U_\varepsilon\}_\varepsilon$, such that 
$$  \nabla U_\varepsilon\to \nabla U_0\ \mbox{a.e. in } \Omega\ \mbox{and } 
T_k (U_\varepsilon)\to T_k (U_0)\ \mbox{(strongly) in } W_0^{1,\overrightarrow{p}}(\Omega)\ \mbox{as } \varepsilon\to 0
$$ for every positive integer $k$. 
\end{lemma}

\begin{proof}
The computations in Subsection~\ref{sect32} can be carried out with $\Phi$ instead of $\Phi_\varepsilon$ since the upper bounds used for $|\Phi_\varepsilon|$ were derived from those satisfied by $|\Phi| $ and the sign-condition of $\Phi$ is the same as for $\Phi_\varepsilon$ (see \eqref{sig}). 
A small change arises in the proof of \eqref{tir} because of the introduction of $f_\varepsilon$ in \eqref{eqn}. Using the definition of $\mathfrak{B}_\varepsilon$ in \eqref{bep}, the inequalities in \eqref{nou} and \eqref{veb} must be read with $\mathfrak{B}_\varepsilon$ instead of $\mathfrak{B}$. 
We note that $\langle \mathfrak{B}_\varepsilon U_\varepsilon,  \varphi_\lambda(z_{\varepsilon,k})\rangle$
is the sum between $\langle \mathfrak{B} U_\varepsilon,  \varphi_\lambda(z_{\varepsilon,k})\rangle$
and $\int_\Omega f_\varepsilon \,\varphi_\lambda(z_{\varepsilon,k})\,dx$.  
The latter term, like the former,  converges to $ 0$ as $\varepsilon\to 0$. The new claim regarding the convergence to zero of
$\int_\Omega f_\varepsilon \,\varphi_\lambda(z_{\varepsilon,k})\,dx$  
follows from the Dominated Convergence Theorem 
using \eqref{nice}, $|\varphi_\lambda(z_{\varepsilon,k})|\leq 2k\exp\,(4\lambda k^2)$ and $\varphi_\lambda(z_{\varepsilon,k})\to 0$ a.e. in $\Omega$ as $\varepsilon\to 0$.  The remainder of the proof of \eqref{tir} carries over easily to our setting. \end{proof}

\subsection{Passing to the limit} \label{limpas} We aim to pass to the limit as $\varepsilon\to 0$ in \eqref{sec} to obtain that $U_0$ is a solution of \eqref{eq101}.  
Since $f_\varepsilon$ satisfies \eqref{conv1} and $U_\varepsilon\rightharpoonup U_0$ (weakly) in 
$W_0^{1,\overrightarrow{p}}(\Omega)$ as $\varepsilon\to 0$, we readily have the 
convergence of the right-hand side of \eqref{sec} to $\langle \mathfrak{B} U_0,v\rangle +\int_\Omega f \,v\,dx$ for every 
$v\in W_0^{1,\overrightarrow{p}}(\Omega)\cap L^\infty(\Omega)$. Moreover, because of the convergence 
$ \nabla U_\varepsilon\to \nabla U_0$ a.e. in $ \Omega$, we can use the same argument as in Lemma~\ref{cor2} 
to deduce that, as $\varepsilon\to 0$, 
\begin{equation*} \label{von2}
\int_\Omega \widehat{\Theta}(u_\varepsilon)\,v\,dx\to  
\int_\Omega \widehat{\Theta}(U_0)\,v\,dx, \quad 
\sum_{j=1}^N \int_\Omega \widehat{A}_j(U_\varepsilon)\,\partial_j v\,dx\to \sum_{j=1}^N \int_\Omega \widehat{A}_j(U_0)\,\partial_j v\,dx 
\end{equation*}
for every $v\in W_0^{1,\overrightarrow{p}}(\Omega)$. 
What is here different compared with Subsection~\ref{pass} is the proof of the convergence
\begin{equation} \label{cov2}
\widehat{\Phi} (U_\varepsilon)\to \widehat{\Phi}(U_0)\ \mbox{ (strongly) in } L^1(\Omega)\ \mbox{as } \varepsilon\to 0.
\end{equation} 
To prove that $U_0$ is a solution of \eqref{eq101}, it remains to justify \eqref{cov2}. 
Instead of Lemma~\ref{coro1}, we establish the following.

\begin{lemma} \label{coro2} We have $\widehat{\Phi}(U_0) \in L^1(\Omega)$ and \eqref{cov2} holds. 
\end{lemma}	

\begin{proof} From Lemma~\ref{cor42}, the pointwise convergence in \eqref{wwa} and the continuity of $\Phi(x,\cdot,\cdot)$ in the last two variables, we infer that
$|\widehat{\Phi}(U_\varepsilon)|\to |\widehat{\Phi}(U_0)|\ \mbox{ a.e. in } \Omega\ \mbox{ as }\varepsilon\to 0.
$
Then, \eqref{cas} and  Fatou's Lemma ensure that $\widehat{\Phi}(U_0) \in L^1(\Omega)$.

\medskip
\noindent {\em Proof of \eqref{cov2}.} 
We will use Vitali's Theorem. To this end,  we need to show that 
$\{\widehat{\Phi}(U_\varepsilon)\}_\varepsilon$ is uniformly integrable over $\Omega$. 
We can only partially imitate the proof of the uniform integrability of 
$\{\widehat{\Phi}_\varepsilon(u_\varepsilon)\}_\varepsilon$ in Lemma~\ref{coro1}. 
Fix $M>1$ arbitrary. For any 
measurable subset $\omega$ of $\Omega$, 
using the growth condition of $\Phi$ in \eqref{cond1}, we find that 
\begin{equation} \label{mini} \int_{\omega} |\widehat{\Phi }(U_\varepsilon)| \,\chi_{\{|U_\varepsilon|\leq M\}} \,dx\leq 
\phi(M)\left(\sum_{j=1}^N  \|\partial_j T_M (U_\varepsilon)\|^{p_j}_{L^{p_j}(\omega)}+
\|c\|_{L^1(\omega)}
\right).
\end{equation}
Since $\partial_j T_M (U_\varepsilon)\to \partial_j T_M( U_0)$ (strongly) in $L^{p_j}(\Omega)$ as $\varepsilon\to 0$ for every $1\leq j\leq N$ and $c\in L^1(\Omega)$, we see that the right-hand side of \eqref{mini} is as small as desired uniformly in $\varepsilon$ when the measure of 
$\omega$ is small. 

\medskip
\noindent We next bound from above $\int_{\omega} |\widehat{\Phi }(U_\varepsilon)| \,\chi_{\{|U_\varepsilon|> M\}} \,dx$. This is where the modification
appears since we don't have anymore that $\{ \widehat{\Phi }(U_\varepsilon) \,U_\varepsilon\}_\varepsilon$ is uniformly bounded in $L^1(\Omega)$ with respect to $\varepsilon$. We adapt an approach from \cite{BOG}.
In \eqref{sec} we take 
$$v=T_1(G_{M-1}(U_\varepsilon)) \in W_0^{1,\overrightarrow{p}}(\Omega)\cap L^\infty(\Omega).$$  Then, using \eqref{newlab0}, the coercivity condition in \eqref{ellip} and the sign-condition of $\Phi$ in \eqref{cond1}, we obtain the estimate 
\begin{equation} \label{vio2}  
\int_\Omega |\widehat{\Phi}(U_\varepsilon)|
\chi_{\{|U_\varepsilon|> M\}}\,dx\leq \int_\Omega (|f_\varepsilon|+C_\Theta)\, \chi_{\{|U_\varepsilon|\geq M-1\}}\,dx+
|\langle \mathfrak{B} U_\varepsilon, T_1(G_{M-1} (U_\varepsilon))\rangle |.
\end{equation}
Now, up to a subsequence of $\{U_\varepsilon\}$, from \eqref{wwa}, we have
$$ T_1(G_{M-1} (U_\varepsilon))\rightharpoonup  
T_1(G_{M-1} (U_0))\ \mbox{(weakly) in } W_0^{1,\overrightarrow{p}}(\Omega)\ \mbox{as } \varepsilon\to 0.
$$ Using this in \eqref{vio2}, jointly with \eqref{nice} and the property $(P_2)$ for $\mathfrak{B}$, we find that
$$ \limsup_{\varepsilon\to 0}
\int_\Omega |\widehat{\Phi}(U_\varepsilon)|
\chi_{\{|U_\varepsilon|> M\}}\,dx\leq 
\int_\Omega (|f|+C_\Theta)\,\chi_{\{ |U_0|\geq M-1\}}\,dx+ 
|\langle \mathfrak{B} U_0, T_1(G_{M-1} (U_0))\rangle |.
$$ 
Recall that $f\in L^1(\Omega)$. 
Since 
$\partial_j \, T_1(G_{M-1}(U_0))=\chi_{\{M-1<|U_0|<M\}} \,\partial_j U_{0}$ a.e. in $\Omega$ for every $1\leq j\leq N$, from the above inequality, we infer that 
$$  \int_{\omega} |\widehat{\Phi }(U_\varepsilon)| \,\chi_{\{|U_\varepsilon|> M\}} \,dx
$$ is small, uniformly in $\varepsilon$ and $\omega$, when $M$ is sufficiently large. Thus, using the comments after \eqref{mini}, we conclude the uniform integrability of $\{\widehat{\Phi}(U_\varepsilon)\}_\varepsilon$ over $\Omega$. The proof of Lemma~\ref{coro2} is complete. \end{proof}

By letting $\varepsilon\to 0$ in \eqref{sec}, we conclude that $U_0$ is a solution of \eqref{eq101}. This ends the proof of Theorem~\ref{nth2} (ii).\qed

\section{Strong convergence of $u_\varepsilon$ in Theorem~\ref{nth2} (i)}
\label{stil} 
We show that in the setting of Theorem~\ref{nth2} (i), up to a subsequence of $\{u_\varepsilon\}$, 
not only the assertions of Lemma~\ref{arc} hold, but also the strong convergence in \eqref{oil}, that is
\begin{equation} 
\label{ola}
u_\varepsilon\to U\ \mbox{(strongly) in } W_0^{1,\overrightarrow{p}}(\Omega) \ \ \mbox{as } \varepsilon\to 0.
\end{equation}
\begin{lemma} \label{stron} 
Up to a subsequence of $\{u_\varepsilon\}_\varepsilon$, 
relabeled $\{u_\varepsilon\}_\varepsilon$, we have \eqref{ola}.
\end{lemma}
\begin{proof}
For every $k\geq 1$, we define 
\begin{equation} \label{lek}  L_k
:=\nu_0^{-1} \left[ |\langle \mathfrak{B} U,G_k(U)\rangle | +C_\Theta\|G_k(U)\|_{L^1(\Omega)}\right]. 
\end{equation}
We next show that, up to a subsequence of $\{u_\varepsilon\}$, we have 
\begin{equation}
\label{duo}
\limsup_{\varepsilon\to 0} \|G_k(u_\varepsilon)\|_{W_0^{1,\overrightarrow{p}}(\Omega)}\leq 
\sum_{j=1}^N L_k ^{1/p_j}.
\end{equation}

\medskip
\noindent \emph{Proof of \eqref{duo}.} Let $k\geq 1$ be a fixed integer.
Since $G_k(u_\varepsilon)=u_\varepsilon-T_k(u_\varepsilon)$ and
$\partial_j T_k(u_\varepsilon)=\partial_j u_\varepsilon \,\chi_{\{|u_\varepsilon|<k\}}$ for every $1\leq j\leq N,$
from the coercivity assumption in \eqref{ellip}, we see that
$$ \begin{aligned}
\langle \mathcal A u_\varepsilon,G_k(u_\varepsilon)\rangle&= \sum_{j=1}^N 
\int_{\{|u_\varepsilon|>k\}} \widehat A_j(u_\varepsilon) \,\partial_j u_\varepsilon\,dx\\
&\geq 
\nu_0 \sum_{j=1}^N \int_{\{|u_\varepsilon|>k\}} |\partial_j u_\varepsilon|^{p_j}\,dx=
\nu_0 \sum_{j=1}^N \|\partial_j G_k(u_\varepsilon)\|^{p_j}_{L^{p_j}(\Omega)}.
\end{aligned}$$
Using \eqref{sig} and $t \,G_k(t)\geq 0$ for every $t\in \mathbb R$, we observe that $G_k(t)\,\widehat{\Phi}_\varepsilon(t)\geq 0$ for all $t\in \mathbb R$. Then, by testing \eqref{ssdili} with $v=G_k(u_\varepsilon)$ and using \eqref{newlab0}, we find that 
$$  \begin{aligned}
\langle \mathcal A u_\varepsilon,G_k(u_\varepsilon)\rangle
&\leq 
\langle \mathcal A u_\varepsilon,G_k(u_\varepsilon)\rangle+\int_\Omega G_k(u_\varepsilon)\,\widehat{\Phi}_\varepsilon(u_\varepsilon)\,dx
\\
&\leq  | \langle \mathfrak B u_\varepsilon,G_k(u_\varepsilon)\rangle |+
C_\Theta \int_\Omega |G_k(u_\varepsilon)|\,dx.
\end{aligned}$$ 
From \eqref{wco}, the boundedness of $\{u_\varepsilon\}$ in $W_0^{1,\overrightarrow{p}}(\Omega)$ and Remark~\ref{an-sob}, we can pass to a subsequence of $\{u_\varepsilon\}$ (relabeled $\{u_\varepsilon\}$) such that, as $\varepsilon\to 0$, we have 
\begin{equation*} \label{tru}
\begin{aligned}
& T_k(u_\varepsilon)\to T_k(U)\ \text{a.e. in }\Omega\ \text{and } 
T_k(u_\varepsilon) \rightharpoonup  T_k(U)\ \text{(weakly) in } W_0^{1,\overrightarrow{p}}(\Omega),\\
& G_k(u_\varepsilon)\to G_k(U)\ \text{a.e. in }\Omega\ \text{and } 
G_k(u_\varepsilon) \rightharpoonup  G_k(U)\ \text{(weakly) in } W_0^{1,\overrightarrow{p}}(\Omega),\\
& G_k(u_\varepsilon) \to  G_k(U)\ \mbox{strongly in }L^{r}(\Omega)\ \mbox{with }1\leq r<p^\ast.
\end{aligned} \end{equation*}
Hence, using the property $(P_2)$, we derive that 
$$  \lim_{\varepsilon\to 0} \langle \mathfrak B u_\varepsilon,G_k(u_\varepsilon)\rangle = \langle \mathfrak B U,G_k(U)\rangle
\quad \mbox{and}\quad \lim_{\varepsilon\to 0} \|G_k(u_\varepsilon)\|_{L^1(\Omega)}= \|G_k(U)\|_{L^1(\Omega)}.
$$
Consequently, for every $1\leq j\leq N$, we have 
$$  \limsup_{\varepsilon \to 0} \|\partial_j (G_k(u_\varepsilon))\|_{L^{p_j}(\Omega)}\leq \left(
\nu_0^{-1}\left[ |\langle \mathfrak{B} U,G_k(U)\rangle  |+C_\Theta\|G_k(U)\|_{L^1(\Omega)}\right]\right) ^{1/p_j}=L_k^{1/p_j}.
$$ 
This establishes the inequality in \eqref{duo}.

\medskip
Recall that $\{u_\varepsilon\}_\varepsilon$ stands for a sequence $\{u_{\varepsilon_\ell}\}_{\ell\geq 1}$ with $\varepsilon_\ell\searrow 0$ as $\ell\to \infty$.  
By Lemma~\ref{l1} and \eqref{duo}, as well as from the proof of  Lemma~\ref{arc}, we get that 
for any given integer $k\geq 1$, there exists a subsequence of 
$\{u_\varepsilon\}_\varepsilon$ that depends on $k$, say $\{u_{\varepsilon_\ell}^{(k)} \}_{\ell\geq 1}$, for which 
\eqref{duo} and \eqref{hsn3} hold with $u_{\varepsilon_\ell}^{(k)}$ in place of $\{u_\varepsilon\}$. This means that 
\begin{equation}
\label{tric}
\limsup_{\ell\to \infty} \|G_k(u_{\varepsilon_\ell}^{(k)})\|_{W_0^{1,\overrightarrow{p}}(\Omega)}\leq \sum_{j=1}^N 
L_k ^{1/p_j},\qquad
\lim_{\ell\to \infty} \|T_k (u_{\varepsilon_\ell}^{(k)}) -T_k (U)\|_{W_0^{1,\overrightarrow{p}}(\Omega)}= 0.
\end{equation} 
We proceed inductively with respect to $k$, at each step $(k+1)$ selecting the subsequence $\{u_{\varepsilon_\ell}^{(k+1)} \}_{\ell\geq 1}$ from 
$\{u_{\varepsilon_\ell}^{(k)} \}_{\ell\geq 1}$, the subsequence of $\{u_\varepsilon\}$ with the properties in \eqref{tric}. Then, $\{u_{\varepsilon_\ell}^{(\ell)}\}_{\ell\geq k}$ is a subsequence of  $\{u_{\varepsilon_\ell}^{(j)} \}_{\ell\geq 1}$ for every $1\leq j\leq k$.
Hence, by a standard diagonal argument, there exists a subsequence 
of $\{u_\varepsilon\}_\varepsilon$, that is, $\{u_{\varepsilon_\ell}^{(\ell)}\}_{\ell}$,    
relabeled $\{u_\varepsilon\}_\varepsilon$, such that \eqref{duo} and \eqref{hsn3} hold for every $k\geq 1$, namely
\begin{equation}
\label{tri}
\limsup_{\varepsilon\to 0} \|G_k(u_\varepsilon)\|_{W_0^{1,\overrightarrow{p}}(\Omega)}\leq \sum_{j=1}^N
L_k ^{1/p_j},\qquad
\lim_{\varepsilon\to 0} \|T_k (u_\varepsilon) -T_k (U)\|_{W_0^{1,\overrightarrow{p}}(\Omega)}= 0.
\end{equation} 
Using the weak convergence of $G_k(u_\varepsilon)$ to $ G_k(U)$
in $W_0^{1,\overrightarrow{p}}(\Omega)$ as $\varepsilon\to 0$, we see that
\begin{equation} \label{folp}
\|G_k(U)\|_{W_0^{1,\overrightarrow{p}}(\Omega)} \leq 
\liminf_{\varepsilon\to 0} \| G_k(u_\varepsilon)\|_{W_0^{1,\overrightarrow{p}}(\Omega)}
\leq 
\sum_{j=1}^N L_k ^{1/p_j}.
\end{equation}
We now complete the proof of \eqref{ola}. From the definition of $G_k$ in \eqref{gk}, we find that 
$$ \| u_\varepsilon-U\|_{W_0^{1,\overrightarrow{p}}(\Omega)} \leq 
\| G_k(u_\varepsilon)\|_{W_0^{1,\overrightarrow{p}}(\Omega)}+
\|G_k(U)\|_{W_0^{1,\overrightarrow{p}}(\Omega)} +
\|T_k (u_\varepsilon) -T_k (U)\|_{W_0^{1,\overrightarrow{p}}(\Omega)}.
$$
Then, in view of \eqref{tri} and \eqref{folp}, for every $k\geq 1$, we obtain that 
\begin{equation} \label{gee}  \limsup_{\varepsilon\to 0} \| u_\varepsilon-U\|_{W_0^{1,\overrightarrow{p}}(\Omega)} \leq 
2  \sum_{j=1}^N L_k ^{1/p_j}.
\end{equation} 
Remark that $L_k$ (defined in \eqref{lek}) converges to $0$ as $k\to \infty$ since $G_k(U)\rightharpoonup 0$ (weakly) in $W_0^{1,\overrightarrow{p}}(\Omega)$ and 
$G_k(U)\to 0$ (strongly) in $L^1(\Omega)$ as $k\to \infty$. Hence, by letting $k\to \infty$ in \eqref{gee}, we obtain \eqref{ola}.  
\end{proof}

{\bf Acknowledgements.} The first author has been supported by the Sydney Mathematical Research Institute via the International Visitor Program (August--September 2019) and by Programma di Scambi Internazionali dell'Universit\`a degli Studi di Napoli Federico II. The research of the second author is supported by the 
Australian Research Council under the Discovery Project Scheme (DP190102948).

\appendix
\section{}
\label{auxre}
In this section, we prove some convergence results that have been used in Sections \ref{newso} and \ref{sec6}, respectively. We assume \eqref{IntroEq0} and \eqref{ellip}.

%For every $r>1$, we denote by $r'$ the conjugate exponent of $r$. 
We first recall an anisotropic Sobolev inequality for the case $p<N$, see \cite{T}.

\begin{lemma}
Let $N\geq 2$ be an integer. If \eqref{IntroEq0} holds, 
then there exists a constant $\mathcal{S}=\mathcal{S}(N,\overrightarrow{p})>0$ such that 
\begin{equation*} \label{send} 
\|u\|_{L^{p^\ast}(\R^N)}\leq \mathcal{S}  \prod_{j=1}^N \| \partial_j  u\|_{L^{p_j}(\R^N)}^{1/N} \quad \mbox{for all } u\in C_c^\infty(\R^N).
\end{equation*}
\end{lemma}

\begin{rem} \label{an-sob} {\rm 
	Let $\Omega$ be a bounded, open subset of $\R^N$ ($N\geq 2$). If  \eqref{IntroEq0} holds, then using a  density argument and the arithmetic-geometric mean inequality, we find that
	\begin{equation}\label{ASI}
	\|u\|_{L^{p^\ast}(\Omega)}\leq \mathcal{S}  \prod_{j=1}^N \| \partial_j  u\|_{L^{p_j}(\Omega)}^{1/N} \leq 
	\frac{\mathcal{S}}{N}\|u\|_{W_0^{1,\overrightarrow{p}}(\Omega)}\quad \mbox{for all } u\in W_0^{1,\overrightarrow{p}}(\Omega).
	\end{equation}
	Moreover,  
	by H\"older's inequality, the embedding $W_0^{1,\overrightarrow{p}}(\Omega)\hookrightarrow L^s(\Omega)$ is continuous 
	for every
	$s\in [1,p^\ast]$ and compact for every $s\in [1,p^\ast)$.}
\end{rem} 

\begin{rem} \label{nra} {\rm 
Note that if $\Omega\subset \mathbb R^N$ is an open bounded domain with Lipschitz boundary and \eqref{IntroEq0} holds, then the ``true" critical exponent is $p_\infty$, the maximum between $p^\ast$ and $p_N$. Indeed, Fragal\`a, Gazzola and Kawohl \cite{FGK} showed that the embedding $W_0^{1,\overrightarrow{p}}(\Omega) \hookrightarrow L^r(\Omega)$ is continuous for every $r\in [1,p_\infty]$ and compact if $r\in [1,p_\infty)$.}
\end{rem}

\subsection{Notation} \label{prel}
For $v,w$ and $\{u_\varepsilon\}_{\varepsilon}$ in $W_0^{1,\overrightarrow{p}}(\Omega)$ and for a.e. $x\in \Omega$, we define
\begin{equation} \label{epslo}
\begin{aligned}
& \mathcal D_{u_\varepsilon}(v,w)(x):=\sum_{j=1}^N \left[
A_j(x,u_\varepsilon(x),\nabla v(x))-A_j(x,u_\varepsilon(x),\nabla w(x))
\right] \partial_j (v-w)(x),\\
& H_{u_\varepsilon}(v,w)(x):= \sum_{j=1}^N A_j(x,u_\varepsilon(x),\nabla v(x))\, \partial_j w(x).
\end{aligned}
\end{equation}
Hence, 
$ \mathcal D_{u_\varepsilon}(v,w)=H_{u_\varepsilon}(v,v)-
H_{u_\varepsilon}(v,w)-H_{u_\varepsilon}(w,v)+H_{u_\varepsilon}(w,w). 
$
The monotonicity assumption in \eqref{ellip} gives that 
$\mathcal D_{u_\varepsilon}(v,w)\geq 0$ a.e. in $\Omega,$ whereas the coercivity condition in \eqref{ellip}
yields that 
$H_{u_\varepsilon}(v,v)\geq \nu_0 \sum_{j=1}^N |\partial_j v|^{p_j},
$
where $\nu_0>0$. We thus find that 
\begin{equation} \label{glan} \mathcal D_{u_\varepsilon}(v,w)\geq \nu_0 \sum_{j=1}^N |\partial_j v|^{p_j}-
|H_{u_\varepsilon}(v,w)|-|H_{u_\varepsilon}(w,v)|. 
\end{equation}

\noindent Here, we establish Lemma~\ref{gnsb1}, which is invoked in the proof of Lemma~\ref{lem-tt11}. Further,  we prove Lemma~\ref{joc1}, which is useful in the proof of Theorem~\ref{nth2} (i) in Section~\ref{sec6}. 
To prove Lemmas~\ref{gnsb1} and \ref{joc1}, we adapt an argument from \cite[Lemma 5]{BMP}, the proof of which goes back to Browder \cite{Bro}. 

\noindent As previously often recalled, by Remark~\ref{an-sob}, whenever 
\begin{equation} 
\label{milsj1}
u_\varepsilon \rightharpoonup u  \ \mbox{(weakly) in } W_0^{1,\overrightarrow{p}}(\Omega)\ \mbox{as } \varepsilon\to 0,	
\end{equation}
we can pass to a subsequence (always relabeled $\{u_\varepsilon\}$) such that 
\begin{equation} \label{gym1} u_\varepsilon\to u\ \mbox{ strongly in } L^r(\Omega) \ \mbox{if } r\in [1,p^\ast)\ \mbox{ and } u_\varepsilon\to u\ \mbox{ a.e. in }\Omega.\end{equation}

\subsection{Some convergence results}

\begin{lemma} \label{gnsb1} Let $u$, $\{u_\varepsilon\}_\varepsilon$ be in $W_0^{1,\overrightarrow{p}}(\Omega)$ such that \eqref{milsj1} holds. If  
$\mathcal D_{u_\varepsilon}(u_\varepsilon,u)\to 0
$ a.e. in $ \Omega$ as $ \varepsilon\to 0$, then, up to a subsequence, 
$\nabla u_\varepsilon\to \nabla u\ \mbox{a.e. in }  \Omega\ \mbox{as } \varepsilon\to 0.  $
\end{lemma}

\begin{proof}
Let $Z$ be a subset of $\Omega$ with $\mbox{meas}\,(Z)=0$ such that for every $x\in \Omega\setminus Z$, we have $|u(x)|<\infty$, $|\nabla u(x)|<\infty$,  
$|\eta_j(x)|<\infty$ for all $1\leq j\leq N$, as well as 
\begin{equation} \label{opak1} u_\varepsilon(x)\to u(x), \quad \mathcal{D}_{u_\varepsilon}(u_\varepsilon,u)(x)\to 0\ \mbox{as } \varepsilon\to 0.
\end{equation} 
For every $x\in \Omega\setminus Z$, we claim that   
\begin{equation} \label{bddk1} 
\{|\nabla u_\varepsilon(x)|\}_{\varepsilon}\ \text{ is uniformly bounded with respect to } \varepsilon. 
\end{equation} 
\emph{Proof of \eqref{bddk1}}.
We fix $x\in \Omega\setminus Z$. In view of \eqref{glan}, we have 
\begin{equation} \label{abi0k1}  
\mathcal D_{u_\varepsilon}(u_\varepsilon,u)(x)\geq \nu_0 \sum_{j=1}^N |\partial_j  u_\varepsilon(x)|^{p_j}-
|H_{u_\varepsilon}(u_\varepsilon,u)(x)|
-|H_{u_\varepsilon}(u, u_\varepsilon)(x)|.
\end{equation}    
By Young's inequality, for every $\delta>0$, there exists $C_\delta>0$ such that 
\begin{equation} \label{abi1m1} \begin{aligned}
& |H_{u_\varepsilon}(u_\varepsilon,u)(x)| \leq 
\sum_{j=1}^N \left(\delta\,  |A_j(x,u_\varepsilon,\nabla  u_\varepsilon)|^{p_j'}+C_\delta  |\partial_j u(x)|^{p_j}\right),\\
& |H_{u_\varepsilon}(u, u_\varepsilon)(x)| \leq 
\sum_{j=1}^N \left(\delta\,  |\partial_j u_\varepsilon(x)|^{p_j} +C_\delta |A_j(x,u_\varepsilon,\nabla u)|^{p_j'}\right).
\end{aligned}
\end{equation}
We use the growth condition  in \eqref{ellip} to bound from above the right-hand side of each inequality in \eqref{abi1m1}. Then, from \eqref{abi0k1}, there exist positive constants $C$ and $\widehat{C_\delta}$, both independent of $\varepsilon$ (with $\widehat{C_\delta}$ depending on $\delta$), such that 
\begin{equation} \label{abi2k1}   \mathcal D_{u_\varepsilon}(u_\varepsilon,u)(x)\geq (\nu_0-C\,\delta) \sum_{j=1}^N |\partial_j u_\varepsilon(x)|^{p_j}- \widehat{C_\delta}\, \mathfrak{g}_{u_\varepsilon}(u)(x),
\end{equation}
where  
$
\mathfrak{g}_{u_\varepsilon}(u)(x)=
\sum_{j=1}^N  \eta_j^{p_j'}(x)+|u_\varepsilon(x)|^{p^\ast} +
\sum_{j=1}^N |\partial_j u(x)|^{p_j}.
$
Using \eqref{opak1} and choosing $\delta\in (0,\nu_0/C)$, from \eqref{abi2k1} 
we conclude \eqref{bddk1}. 

\medskip
\noindent \emph{Proof of Lemma \ref{gnsb1} concluded.} 
Let $x\in \Omega\setminus Z$ be arbitrary. Define 
$\xi_\varepsilon=\nabla u_\varepsilon(x)\ \mbox{ and } \xi=\nabla  u(x).$ 
To show that $\xi_\varepsilon\to \xi$ as $\varepsilon\to 0$, it is enough to prove that any accumulation point of $\xi_\varepsilon$, say
$\xi^*$, coincides with $\xi$. From \eqref{bddk1}, we have $|\xi^*|<\infty$. 
By \eqref{opak1} and the continuity of $ A_j(x,\cdot,\cdot)$ with respect to the last two variables, we find that 
$$ {\mathcal D}_{u_\varepsilon}(u_\varepsilon,u)(x)\to  \sum_{j=1}^N \left[A_j(x,u(x),\xi^*)-  A_j(x,u(x),\xi)\right] (\xi^*_j- \xi_j)\quad \text{as } \varepsilon\to 0.   
$$ 
This, jointly with \eqref{opak1} and the monotonicity condition in \eqref{ellip}, gives that 
$\xi^*=\xi$. This ends the proof since $x\in \Omega\setminus Z$ is arbitrary and $\mbox{meas}\,(Z)=0$.  
\end{proof}

\begin{lemma} \label{joc1} Let $k\geq 1$ be a fixed integer. 
Let $u$, $\{u_\varepsilon\}_\varepsilon$ be in $W_0^{1,\overrightarrow{p}}(\Omega)$ such that \eqref{milsj1} holds. 
Suppose that, up to a subsequence of $\{u_\varepsilon\}$ (depending on $k$ and relabeled $\{u_\varepsilon\}$) 
\begin{equation} \label{amil1} \mathcal D_{u_\varepsilon}(T_k(u_\varepsilon),T_k(u))\to 0
\quad \mbox{ in } L^1(\Omega)\ \mbox{ as } \varepsilon\to 0.
\end{equation}  
Then, up to a subsequence of $\{u_\varepsilon\}$, as $\varepsilon\to 0$, we have 
\begin{eqnarray}
& \nabla T_k (u_\varepsilon)\to \nabla T_k (u)\ \mbox{a.e. in } \Omega, \label{hsn11}\\
&  T_k (u_\varepsilon)\to T_k (u)\ \mbox{(strongly) in } W_0^{1,\overrightarrow{p}}(\Omega).\label{hsn21}
\end{eqnarray}
\end{lemma}

\begin{proof} 
By \eqref{milsj1} and \eqref{amil1}, up to a subsequence of $\{u_\varepsilon\}$, we have \eqref{gym1}, as well as 
$\mathcal D_{u_\varepsilon}(T_k(u_\varepsilon),T_k(u))\to 0$ a.e. in $\Omega$ as  $\varepsilon\to 0$. 
Let $Z$ be a subset of $\Omega$ as in the proof of Lemma~\ref{gnsb1}, where 
$\mathcal D_{u_\varepsilon}(T_k(u_\varepsilon),T_k(u)))$ replaces $\mathcal D_{u_\varepsilon}(u_\varepsilon,u)$. We follow the same argument as in Lemma~\ref{gnsb1} with the obvious modifications suggested by the above replacement. Then, for every $x\in \Omega\setminus Z$, we obtain  
\begin{equation} \label{abi01} 
\begin{aligned} 
\mathcal D_{u_\varepsilon}(T_k(u_\varepsilon),T_k(u))(x)\geq  
& \nu_0 \sum_{j=1}^N |\partial_j  T_k(u_\varepsilon)(x)|^{p_j}\\
&
-	|H_{u_\varepsilon}(T_k(u_\varepsilon),T_k(u))(x)|
-|H_{u_\varepsilon}(T_k(u), T_k(u_\varepsilon))(x)|.
\end{aligned}
\end{equation} 
This leads to 
$ \{|\nabla T_k (u_\varepsilon)(x)|\}_{\varepsilon}$ being uniformly bounded with respect to $ \varepsilon$ and we also obtain \eqref{hsn11}.

\vspace{0.2cm}
We conclude the proof of Lemma~\ref{joc1} by showing \eqref{hsn21}. 
From \eqref{hsn11}, we see that 
$\{|\partial_j T_k (u_\varepsilon)-  \partial_j T_k (u)|^{p_j}\}_\varepsilon$ is a sequence of nonnegative integrable functions, converging to $0$ a.e. on $\Omega$. Thus, by Vitali's Theorem, we obtain that $\partial_j T_k (u_\varepsilon)\to \partial_j T_k (u)$ in $L^{p_j}(\Omega)$ as $\varepsilon\to 0$ for every $1\leq j\leq N$ by proving that   	
\begin{equation} \label{unif1} \left\{\sum_{j=1}^N |\partial_j T_k (u_\varepsilon)|^{p_j}\right\}_{\varepsilon}\ \mbox{ is uniformly integrable over } \Omega. \end{equation} 
The claim of \eqref{unif1} follows 
from \eqref{amil1} and \eqref{abi01} whenever $ \{H_{u_\varepsilon}( T_k(u_\varepsilon),T_k(u))\}_\varepsilon $ and 
$\{H_{u_\varepsilon}(T_k(u), T_k(u_\varepsilon))\}_\varepsilon$ converge in $L^1(\Omega)$ as $\varepsilon\to 0$. 	We next establish that 
\begin{equation} \label{coni11} 
\begin{aligned}
H_{u_\varepsilon}(T_k(u_\varepsilon),T_k(u))& \to \sum_{j=1}^N A_j(x,u,\nabla T_k (u))\,\partial_j T_k(u)\quad \mbox{in } L^1(\Omega)\ \mbox{as } \varepsilon\to 0,
\\
H_{u_\varepsilon}(T_k(u),T_k(u_\varepsilon)) &\to \sum_{j=1}^N A_j(x,u,\nabla T_k (u))\,\partial_j T_k(u)\quad \mbox{in } L^1(\Omega)\ \mbox{as } \varepsilon\to 0.  
\end{aligned}
\end{equation}

\medskip
\noindent 	{\em Proof of \eqref{coni11}.} 
Let $1\leq j\leq N$ be arbitrary. We see that $\{A_j(x,u_\varepsilon,\nabla T_k (u_\varepsilon))\}_\varepsilon$ is bounded in $L^{p_j'}(\Omega)$ from the growth condition in \eqref{ellip} and the boundedness of $\{u_\varepsilon\}_{\varepsilon}$ in $W_0^{1,\overrightarrow{p}}(\Omega)$ and, hence, in $L^{p^\ast}(\Omega)$. 
Moreover, $A_j(x,u_\varepsilon,\nabla T_k (u_\varepsilon))\to A_j(x,u,\nabla T_k (u))$ a.e. in $ \Omega$ as $\varepsilon\to 0$ using \eqref{hsn11}, the convergence $u_\varepsilon\to u$ a.e. in $\Omega$ (from \eqref{gym1}) and the continuity of $A_j(x,\cdot,\cdot)$ in the last two variables. Thus, up to a subsequence of $\{u_\varepsilon\}$, we infer that $ A_j(x,u_\varepsilon,\nabla T_k (u_\varepsilon))
\rightharpoonup A_j(x,u,\nabla T_k (u))$ (weakly) in $L^{p_j'}(\Omega)$ as $\varepsilon\to 0.$
This proves the first convergence in \eqref{coni11}.  We now prove the second one.

Using \eqref{hsn11} and the continuity properties of $A_j$, as $\varepsilon\to 0$,
\begin{equation} \label{vit0} A_j(x,u_\varepsilon,\nabla T_k (u))\,\partial_j T_k (u_\varepsilon)\to 
A_j(x,u,\nabla T_k (u))\,\partial_j T_k (u)\ \mbox{a.e. in }  \Omega \end{equation}
for each $1\leq j\le N$.
Observe that $\{ \chi_{\{|u_\varepsilon|<k\} } |A_j(x,u_\varepsilon,\nabla T_k (u))|^{p_j'}\}_\varepsilon 
$ is uniformly integrable over $\Omega$ (from the growth condition of $A_j$ in  \eqref{ellip}) and 
$\partial_j T_k (u_\varepsilon)= \chi_{\{|u_\varepsilon|<k\} } \,\partial_j u_\varepsilon.$ 
Thus, since 
$\{\partial_j u_\varepsilon\}_\varepsilon$ is bounded in $L^{p_j}(\Omega)$, it follows from H\"older's inequality that 
$\{ A_j(x,u_\varepsilon,\nabla T_k (u))\,\partial_j T_k (u_\varepsilon)\}_\varepsilon
$ is uniformly integrable over $\Omega$
for each $1\leq j\leq N$. From \eqref{vit0} and  Vitali's Theorem, we reach the claim of \eqref{coni11}.   
\end{proof}

From Lemma~\ref{joc1} and a standard diagonal argument, we  obtain the following. 

\begin{cor}
Let \eqref{milsj1} and \eqref{amil1} hold. 
Then, there exists a subsequence of $\{u_\varepsilon\}_\varepsilon$, 
relabeled $\{u_\varepsilon\}_\varepsilon$, such that 
$  \nabla u_\varepsilon\to \nabla u$ a.e. in $\Omega$ and 
$T_k (u_\varepsilon)\to T_k (u)$ (strongly) in  $W_0^{1,\overrightarrow{p}}(\Omega)$ as $ \varepsilon\to 0
$ for every integer $k\geq 1$. 
\end{cor}

\end{document}